\documentclass[12pt]{amsart}

\usepackage{amsmath,amssymb,latexsym}
\usepackage[a4paper,total={16cm,23.5cm},top=3cm,left=2.5cm]{geometry}
\usepackage{hyperref}

\def\im{\mathrm{im}}

\theoremstyle{plain}

\newtheorem{theorem}{Theorem}[section]
\newtheorem{proposition}[theorem]{Proposition}
\newtheorem{notation}[theorem]{Notation}

\newtheorem{fact}[theorem]{Fact}
\newtheorem{fact/def}[theorem]{Fact/Definition}
\newtheorem{lemma}[theorem]{Lemma}
\newtheorem{example}[theorem]{Example}
\newtheorem{corollary}[theorem]{Corollary}
\newtheorem{claim}{Claim}

\theoremstyle{definition}
\newtheorem{definition}[theorem]{Definition}
\newtheorem{remark}[theorem]{Remark}
\newtheorem{question}[theorem]{Question}

\newcommand{\be}{\begin{enumerate}}
\newcommand{\bi}{\begin{itemize}}

\newcommand{\bd}{\begin{definition}}

\newcommand{\bt}{\begin{theorem}}
\newcommand{\bl}{\begin{lemma}}
\newcommand{\bc}{\begin{corollary}}
\newcommand{\bft}{\begin{fact}}
\newcommand{\bp}{\begin{proposition}}
\newcommand{\br}{\begin{remark}}
\newcommand{\er}{\end{remark}}
\newcommand{\ep}{\end{proposition}}
\newcommand{\ef}{\end{fact}}
\newcommand{\ec}{\end{corollary}}
\newcommand{\ee}{\end{enumerate}}
\newcommand{\ei}{\end{itemize}}
\newcommand{\ed}{\end{definition}}
\newcommand{\et}{\end{theorem}}
\newcommand{\el}{\end{lemma}}
\newcommand{\bpf}{\begin{proof}}
\newcommand{\bpfc}{\begin{proof}[Proof of the claim]}
\newcommand{\epf}{\end{proof}}

\def\Ind#1#2{#1\setbox0=\hbox{$#1x$}\kern\wd0\hbox to 0pt{\hss$#1\mid$\hss}
\lower.9\ht0\hbox to 0pt{\hss$#1\smile$\hss}\kern\wd0}
\def\ind{\mathop{\mathpalette\Ind{}}}
\def\Notind#1#2{#1\setbox0=\hbox{$#1x$}\kern\wd0\hbox to 0pt{\mathchardef
\nn="3236\hss$#1\nn$\kern1.4\wd0\hss}\hbox to 0pt{\hss$#1\mid$\hss}\lower.9\ht0
\hbox to 0pt{\hss$#1\smile$\hss}\kern\wd0}
\def\nind{\mathop{\mathpalette\Notind{}}}
\def\lin{\mathrm{lin}}
\def\Dim{\mathrm{Dim}}
\def\Codim{\mathrm{Codim}}
\def\rk{\mathrm{rk}}
\def\acl{\mathrm{acl}}
\def\Mlt{\mathrm{Mlt}}
\def\DM{\mathrm{DM}}
\def\RM{\mathrm{RM}}
\def\dcl{\mathrm{dcl}}

\def\Lin{\mathrm{Lin}}

\def\tp{\mbox{tp}}

\def\id{\mathrm{id}}

\def\tp{\mathrm{tp}}

\def\C{\mathfrak{C}}

\begin{document}
\title{Sets, groups, and fields definable in vector spaces with a bilinear form}
\author{Jan Dobrowolski}
\address{Instytut Matematyczny, Uniwersytetu Wroc\l awskiego, pl.\ Grunwaldzki 2/4, 50-383 Wroc\l aw\newline
\indent {\em and}\newline
\indent School of Mathematics, University of Leeds, Leeds LS2 9JT, UK}
\email{dobrowol@math.uni.wroc.pl}

\begin{abstract}We study definable sets, groups, and fields in the theory $T_\infty$ of infinite-dimensional vector spaces over an algebraically closed field of any fixed  characteristic different from 2 equipped with a nondegenerate symmetric (or alternating) bilinear form. First, we define an ($\mathbb{N}\times \mathbb{Z},\leq_{lex}$)-valued dimension on definable sets in $T_\infty$ enjoying many properties of Morley rank in strongly minimal theories. Then, using this dimension notion as the main tool, we prove that all groups definable in $T_\infty$ are (algebraic-by-abelian)-by-algebraic, which, in particular, answers a question from \cite{Gr}. We conclude that every infinite field definable in $T_\infty$ is definably isomorphic to the field of scalars of the vector space.
We derive some other consequences of good behaviour of the dimension in $T_\infty$, e.g. every generic type in any definable set is a definable type; every set is an extension base; every definable group has a definable connected component.

We also consider the theory $T^{RCF}_\infty$ of vector spaces over a real closed field equipped with a nondegenerate alternating bilinear form or a nondegenerate symmetric positive-definite bilinear form. As in the case of $T_\infty$, we define a dimension on sets definable in $T^{RCF}_\infty$, and using it we prove analogous results about definable groups and fields: every group definable in $T^{RCF}_{\infty}$ is (semialgebraic-by-abelian)-by-semialgebraic (in particular, it is (Lie-by-abelian)-by-Lie), and every field definable in $T^{RCF}_{\infty}$ is definable in the field of scalars, hence it is either finite or real closed or algebraically closed.
\end{abstract}
\maketitle
\tableofcontents

\section{Introduction}
There are two kinds of motivation for the study undertaken in this paper.

The first is improving our understanding of definable sets and other definable objects (such as groups and fields) in classical mathematical structures. 
There is a variety of this kind of results in numerous contexts; we mention few of them. 
In algebraically closed fields there is a very well-behaved notion of dimension on definable sets (given by the algebraic dimension of the Zariski closure of a set, which coincides with a more general notion of Morley rank) and the following well-known description of definable groups and fields follows from results by Weil, Hrushovski, and van den Dries (see \cite{Wei},\cite{van1},\cite{van2},\cite{Bou}).
\begin{fact}\label{fact_acf}
Let $K$ be an algebraically closed field. Then:\\
 (1) The groups definable in $K$ are precisely the algebraic groups over $K$.\\
 (2) Every field definable in $K$ is definably isomorphic to $K$.
\end{fact}
Variants of these statements for separably closed fields were proved in \cite{Mes}.
In the real closed fields and their $o$-minimal expansions, again, there is a very nice notion of dimension, and Pillay's Conjecture provides a link between definable groups and Lie groups. Moreover, the following was proved in (\cite{Pil2}).
\begin{fact}\label{fact_omin_fields}
 Every infinite field definable in an o-minimal structure is either real closed or algebraically closed.
\end{fact}

There are many more results on groups definable in fields and in their expansions such as differential fields, fields with a generic automorphism, or valued fields.
In a  different flavour, it was proved in \cite{skl} that there are no infinite  fields definable in free groups.
Groups definable in ordered vector spaces over  ordered division rings were studied in \cite{ES}.  

Our second motivation is understanding certain phenomena in NSOP$_1$ structures - a very broad class of `tame' structures	 studied intensively in recent years, with the vector spaces with a generic bilinear form being one of the main algebraic examples. This motivation is addressed most directly in Section \ref{sect_generics}, which, however, relies on our study of dimension in earlier sections.

A systematic study of vector spaces with a bilinear form was first undertaken in \cite{Gr}. Several fundamental results concerning completeness, model completeness, and quantifier elimination were established there. As finite-dimensional vector spaces with a bilinear form are definable in the underlying field of scalars, only the infinite-dimensional case goes really beyond the (model-theoretic) study of the field. The main focus in \cite{Gr} was on the theory $T_\infty$ of infinite-dimensional vector spaces over an algebraically closed field of any fixed characteristic different from 2 with a nondegenerate symmetric (or alternating) bilinear form (this is 
a slight abuse of notation, as this means in  fact considering a family of different theories, depending on whether the form is assumed to be symmetric or alternating, and also on the characteristic of the field of scalars, all of which are denoted by $T_\infty$).
A certain independence relation $\ind^\Gamma$ on models on $T_\infty$ was constructed there, and it was proved that it shares many nice properties with forking independence in stable theories (forking independence is a central notion in model theory generalising linear independence in vector spaces and algebraic independence in algebraically closed fields to abstract contexts). These results were later used in \cite{CR} to prove that $T_\infty$ is NSOP$_1$.  $T_\infty$ was further studied in \cite{KR}, where the canonical independence relation in NSOP$_1$ theories called Kim-independence (and denoted $\ind^K$) was introduced, and described in particular in $T_\infty$ (some corrections are needed in that description, see Proposition \ref{kim_description} and the discussion preceding it). 
It was then deduced in  \cite{KR} that  $\ind^\Gamma$ is strictly stronger than $\ind^K$.

In \cite{CH} it was proved that (the completions of) the theories of vector spaces with a nondegenerate bilinear form over an NIP (another tameness property studied extensively in model theory) field satisfy a generalisation of NIP called NIP$_2$; in particular, $T_\infty$ and $T^{RCF}_\infty$ (see the paragraph below) are examples of  NIP$_2$ theories which are not NIP. 

In this paper, we study the theory (strictly speaking, the theories) $T_\infty$ and the theory (two theories) $T^{RCF}_\infty$ of infinite-dimensional vector spaces over a real closed field equipped with a nondegenerate alternating bilinear form or a nondegenerate symmetric positive-definite bilinear form (RCF stands for the theory of real closed fields). In the final chapter of \cite{Gr} (12.5) it was asked whether every group definable in $T_\infty$ is finite Morley rank-by-abelian-by finite Morley rank, which we confirm in Section \ref{sec_gps}. We also prove that finite Morley rank quotients of groups definable in  $T_\infty$ by definable normal subgroups are definable in $T_\infty$ and hence they are algebraic (by `algebraic' we mean definably isomorphic to an algebraic group over the field of scalars), thus  obtaining that all groups definable in $T_\infty$ are (algebraic-by-abelian)-by-algebraic. This conclusion is optimal in the sense that none of the three components in `(algebraic-by-abelian)-by-algebraic' can be omitted (see Remark \ref{optimal}).
Using our theorem about groups, we deduce that every field definable in $T_\infty$ is finite-dimensional, and hence either finite or definably isomorphic to the field of scalars. We also prove analogous results about groups and fields definable in $T^{RCF}_\infty$. As our main tool, we develop a notion of dimension on sets definable in $T_\infty$ and $T_\infty^{RCF}$, whose good behaviour has several other consequences which may be of independent interest.

Most of the arguments in the paper are carried out simultaneously for $T_\infty$, where we use Morley rank to define dimension, and for $T^{RCF}_\infty$, where we use a topological dimension (called \emph{$o$-minimal dimension}) for this purpose. Except Section \ref{sect_generics} where we focus on model-theoretic properties of $T_\infty$, the only  significant difference between the two cases is that in $T_\infty$ every definable set has finite multiplicity with respect to our dimension notion, which does not hold in $T^{RCF}_\infty$. Because of this, we need separate arguments for $T_\infty$ and  $T^{RCF}_\infty$ in  the proof of Corollary \ref{cor_elim}. 
Our proof of finiteness of multiplicity in $T_\infty$ implies in particular that given a system of finitely many equations using the linear space operations and the bilinear form, the algebraic varieties obtained by intersecting the set of solutions of the system with finite-dimensional nondegenerate linear subspaces have uniformly bounded number of irreducible components of maximal dimension in the sense of algebraic geometry (cf. Theorem \ref{thm_mlt}(1)).


The paper is organised as follows. In Section \ref{prel} we recall some basic facts about Morley rank and the $o$-minimal (topological) dimension, and about model theory of vector spaces with bilinear forms. 

In Section \ref{sec_1} we review the notions of dimension and codimension of a definable subset of the vector sort $V$ introduced in \cite{Gr}, filling a gap in the construction.

In Section \ref{sect_dim} we extend the notion of dimension to arbitrary definable sets and types in $T_\infty$ and $T^{RCF}_\infty$, and we prove that it has properties similar to those of Morley rank in strongly minimal theories (Corollary \ref{big_col}).

In Section \ref{sect_lascar} we prove an analogue of Lascar's equality for $T_\infty$ and $T^{RCF}_\infty$, and we relate our notion of dimension to the linear dimension.

In Section \ref{sect_mlt} we define multiplicity of a definable set in analogy with Morley degree, and we prove that every set definable in $T_\infty$ has finite multiplicity. Using this, we prove that a quotient of a group definable in $T_\infty$ by a definable normal subgroup is algebraic provided that it has finite Morley rank (and, using some additional argument, an analogous result for $T^{RCF}_\infty$). We also derive another consequence of finiteness of multiplicity in $T_\infty$: in every definable set there are only finitely many complete generic types (over any fixed model), and each of them is a definable type.

In Section \ref{sec_gps} we first observe that every group definable in $T_\infty$ has a definable connected component, and then we prove the main results of this paper: every group definable in $T_\infty$ is (algebraic-by-abelian)-by-algebraic, and every field definable in $T_\infty$ has finite dimension,  hence is either finite or definably isomorphic to the field of scalars $K$. Simultaneously, we prove the corresponding results for $T^{RCF}_\infty$.

In Section \ref{sect_generics} we prove that every set of parameters in $T_\infty$ is an extension base (i.e. $T_\infty$ satisfies the existence axiom for forking independence) and we give a description of Kim-independence in $T_\infty$ over arbitrary sets, correcting in particular the description of Kim-independence over models in $T_\infty$ given in \cite{KR}. Finally, we prove that in every group $G$ definable in $T_\infty$  the $\ind^\Gamma$-generics are precisely the generics in the sense of dimension (in particular $\ind^\Gamma$-generics exist in $G$), and that the additive group $(V,+)$ of the vector sort does not have any $\ind^K$-generics over any set.

All sections except the last one (Section \ref{sect_generics}) require only a very basic understanding of first-order logic, and should be accessible to readers familiar with concepts such as a model, a complete theory, a type (i.e. a consistent set of formulas), quantifier elimination. 

The author thanks Ehud Hrushovski for pointing out Example \ref{examp}(2) to him, Nick Ramsey for a discussion about Kim-independence in $T_\infty$, and the logic group in Leeds for helpful comments during his seminar talk reporting on this work, 

\section{Preliminaries}\label{prel}
\subsection{Morley rank and the $o$-minimal dimension}
Let $T$ be a complete theory, and let $M\models T$.
\begin{definition}
 The Morley rank of a formula $\phi$ over $M$ defining a set $S$, denoted $\RM(\phi)$ or $\RM(S)$, is an ordinal or $-1$ or $\infty$, defined by first recursively defining what it means for a formula to have Morley rank at least $\alpha$ for some ordinal $\alpha$:
\begin{itemize}
 \item $\RM(S)\geq 0$ iff $S\neq  \emptyset$.
 \item If $\alpha=\beta+1$ is a successor, then $\RM(S)\geq \alpha$ iff for every $n\in \omega$ there are disjoint sets $(X_i)_{i\in \{1,\dots,n\}}$ definable in some elementary extension $N$ of $M$ such that $\RM(X_i)\geq \beta$ and $X_i\subseteq \phi(N)$ for each $i\in \{1,\dots,n\}$.
 \item If $\lambda$ is a limit ordinal then $\RM(S)\geq \lambda$ iff $\RM(S)\geq \alpha$ for every $\alpha<\lambda$.
\end{itemize}

Finally, $\RM(S)=\alpha$ when $\RM(S)\geq \alpha$ and for no $\beta>\alpha$
one has $\RM(S)\geq \beta$. Also, we set $\RM(S)=\infty$ if $\RM(S)\geq \alpha$ for every $\alpha\in Ord$. If $\RM(S)\in Ord$, then the Morley degree of $S$, denoted by $\DM(S)$, is the maximal number of definable sets of Morley rank $\RM(S)$ into which $S$ can be partitioned.
\end{definition}
If $M=(F,+,\cdot,0,1)$ is an algebraically closed field (which is essentially the only case in which we consider Morley rank in this paper), passing to an elementary extension $N$ of $M$ in the above definition is not necessary - the sets $X_i$  may be chosen to be definable in $M$.

A one-sorted structure $M$ (or its theory $Th(M)$) is called \emph{strongly minimal} if $\RM(x=x)=\DM(x=x)=1$ where $x$ is a single variable of the only sort of $M$. Equivalently, every definable subset of any model $\C\models Th(M)$ is either finite of co-finite. Any algebraically closed field is strongly minimal. For $p$ equal zero or a prime number, $ACF_p$ denotes the (complete) theory of algebraically closed fields of characteristic $p$.

\begin{fact}\cite{Mac}\label{fact_mac}
 An infinite field has finite Morley rank if and only if it is algebraically
closed (if and only if it is strongly minimal).
\end{fact}

If $K$ is an algebraically closed field and $X$ is an algebraic subset of $K^n$ for some $n\in \omega$, then $\RM(X)$ is the dimension of $X$ in the sense of algebraic geometry, and $\DM(X)$ is the number of irreducible components of $X$ of maximal dimension.


In real closed fields Morley rank of any infinite set is equal to $\infty$, but there is another useful notion of dimension (having various equivalent definitions). 
\begin{definition}
 Let $(R,+,\cdot,\leq)$ be a real closed field (or, more generally, an $o$-minimal structure).  
For  a nonempty definable $X\subseteq R^k$  the (topological) dimension of $X$, denoted by $\dim_t(X)$, is the greatest number $n$ such that a nonempty definable open (in the order topology) subset of $R^n$ embeds definably into $X$. We also put $\dim_t(\emptyset)=-1$.
 \end{definition}
Again, for algebraic subsets of $R^n$ where $R\models RCF$, $\dim_t$ coincides with the dimension in the sense of algebraic geometry.

\begin{definition}
 We say an $S$-valued (where $S$ is any set) rank $\rk$ on the collection of all  sets definable in $T$ is \emph{definable} (over $\emptyset$) if for any formula $\phi(x,y)$ over $\emptyset$, $n\in S$, and $\C\models T$ the set $\{a\in \C:\rk(\phi(x,a))=n\}$ is definable over $\emptyset$. 
\end{definition}

\begin{fact}\label{fact_rm}
Let $\rk$ be either Morley rank in a strongly minimal theory, or the topological dimension in a real closed field (or in any $o$-minimal theory).
Suppose $X_1$ and $X_2$ are definable. Then:\\ 
(0) $\rk(X_1)\in \omega\cup \{-1\}$ and $\rk(X_1)=0$ iff $X_1$ is finite and nonempty.\\
(1) If $X_1\subseteq X_2$, then $\rk(X_1)\leq \rk(X_2)$.\\
(2) $\rk(X_1\cup X_2)=\max(\rk(X_1),\rk(X_2))$.\\
(3) If there is a definable bijection between $X_1$ and $X_2$, then $\rk(X_1)=\rk(X_2)$.
\\
(4) More generally, if  $f:X_1\to X_2$ is a definable surjection and there is $d\in \omega$ is such that $\rk(f^{-1}(y))=d$ for each $y\in X_2$, then $\rk(X_1)=\rk(X_2)+d$ unless $X_2$ is empty.\\
In particular, if $\emptyset\neq Z\subseteq Y\times X$ and there is $d\in \omega$  such that $\rk(\{x\in X:(y,x)\in Z\})=d$ for every $y\in Y$, then $\rk(Z)=\rk(Y)+d$. \\
(5) $\rk$ is definable over $\emptyset$.

In a strongly minimal theory we also have:\\
(6)  For any $n,d\in \omega$ and a formula $\phi(x,y)$  the set $\{a\in \mathbb{\C}:\RM(\phi(x,a))=n,\DM(\phi(x,a))=d\}$ is definable over $\emptyset$, and only for finitely many pairs $(n,d)$ this set is nonempty.\\
(7) If $\RM(X_1)<\RM(X_2)$, then $\DM(X_1\cup X_2)=\DM(X_2)$.\\
 (8) $\DM(X_1\cup X_2)\leq \DM(X_1)+\DM(X_2)$\\
 (9) If $f:X\to Y$ is a definable surjection such that $\DM(Y)=m\in \omega$ and there are $s,m'\in \omega$ such that $\RM(f^{-1}(y))=s$ and $\DM(f^{-1}(y))\leq m'$ for every $y\in Y$, then $\DM(Y)\leq \DM(X)\leq mm'$.
\end{fact}
(9) above can be proved in the same way as Proposition \ref{prop_mlt}(4).

Finally, let us mention that if $K$ is an algebraically closed field or a real closed field, then it admits (uniform) \emph{elimination of imaginaries} (EI), that is, if $E$ is any definable equivalence relation on $K^n$ then the quotient $K^n/E$ is in a definable bijection (in the structure $K$ with the sort $K^n/E$ added) with a definable subset of $K^m$ for some $m$. However, the theories $T_\infty$ and $T^{RCF}_\infty$ considered in this paper do not admit EI, hence we will need some extra care when dealing with quotients  there. 

%

\subsection{Generic bilinear forms}
We start by recalling some notation from \cite{Gr}. 	
Let $L$ be the two-sorted language with sorts $V$ (vectors), and $K$ (scalars), containing constant symbols $0_V$, $0_K$, and $1_K$, as well as binary function symbols: $+_V$, $+_K$, $\circ_K$, $\gamma$, $[,]$, which we shall interpret as: 
vector addition, field addition, field multiplication, scalar multiplication, and a bilinear form on the vector space.

We fix $p$ to be $0$ or a prime number different from $2$, and we let $T_0=ACF_p$, the (complete) theory of algebraically closed fields of characteristic $p$. As the value of $p$ will not play any role in the paper (and in the results of \cite{Gr}), it is omitted in the notation below.  

\begin{definition}
Let  $m\in \omega\cup \{\infty\}$ and $T_0$ be either ACF$_p$ or RCF. By $_ST^{T_0}_m$ [respectively, $_AT^{T_0}_m$] we denote the $L$-theory expressing that the sort $K$ is a model of $T_0$, the sort $V$ is an $m$-dimensional vector space over $K$, and that $[,]$ is a nondegenerate symmetric [respectively, alternating] $K$-bilinear form on $V$, and additionally $_ST_m^{RCF}$ says that $[,]$ is positive-definite. We will write  $T^{T_0}_m$ to mean either $_ST^{T_0}_m$ or $_AT^{T_0}_m$. We will also simply write $T_m$ to mean $T^{ACF_p}_m$, which is consistent with \cite[Chapter 12]{Gr}, and  $T^*_m$ to mean either
$T^{ACF_p}_m$ or $T^{RCF}_m$.

If $m\in \omega$ then $_ST^{T_0}_m$ is consistent only when $m$ is even, so below we will always assume that $m=\infty$ or $m$ is even in the symmetric case.
\end{definition}

\begin{definition}
 For any $n<\omega$ let $\theta_n(X_1,\dots,X_n)$ be the $L$-formula saying that the vectors $X_1,\dots, X_n$ are linearly independent. Let $L_\theta$ be the expansion of $L$ obtained by adding to $L$ a symbol $\theta_n$ for each $n$ (which we shall interpret as the relation given by the formula $\theta_n$).
\end{definition}

For any $n\in \{1,2,\dots \}$ let $F_n:V^{n+1}\to K^n$ be a definable function sending any tuple $(v_1,\dots,v_{n+1})$ with  $v_1,\dots,v_n$ linearly independent and $v_{n+1}\in \Lin_{K}(v_1,\dots,v_n)$ to the unique tuple $(a_1,\dots,a_n)\in K^n$ such that $v_{n+1}=a_1v_1+\dots+a_nv_n$ (and any other tuple to $(0_K,\dots,0_K)$, say).
In  \cite[Corollary 9.2.3]{Gr} Granger claimed that $T^{T_0}_m$ has quantifier elimination in the language $L_\theta\cup L_K$, where $L_K$ is any language on $K$ bidefinable with $(K,+,\cdot)$ in which $K$ has quantifier elimination.
D. MacPherson has later pointed out that there is a problem with this result, unless one adds function symbols for each $F_n$ to the language. A. Chernikov and N. Hempel have proved that indeed $T^{T_0}_m$  eliminates quantifiers in $L_\theta\cup L_K\cup\{F_n:n\in\omega\}$. Let us remark here that, in the symmetric positive-definite case over a real closed field, the functions $F_n$ are equal to some terms in the language $L_\theta$, hence adding the  $F_n$'s to the language is necessary only in the alternating case. 
For let $v_1,\dots,v_n\in V$ be linearly independent, and $v_{n+1}=\Sigma_{i\leq n} a_iv_i$ for some $a_1,\dots,a_n\in K$.
Let $A$ be the $n\times n$-matrix $([v_i,v_j])_{i,j\leq n}$ and note that $A(a_1,\dots,a_n)^T=([v_1,v_{n+1}],\dots,[v_n,v_{n+1}])^T$. Note that if $b_1,\dots,b_n$ is such that $A(b_1,\dots,b_n)^T=([v_1,v_{n+1}],\dots,[v_n,v_{n+1}])^T$ then $(b_1,\dots,b_n)=(a_1,\dots,a_n)$ as otherwise $\Sigma_{i\leq n}(a_i-b_i)v_i$ would be a non-zero vector orthogonal to $v_1,\dots,v_n$, hence orthogonal to itself, which is a contradiction. So the equation $A(x_1,\dots,x_n)^T=([v_1,v_{n+1}],\dots,[v_n,v_{n+1}])^T$ has exactly one solution $(x_1,\dots,x_n)=(a_1,\dots,a_n)$, and so $A$ is a non-singular matrix and $$F(v_1,\dots,v_{n+1})=(a_1,\dots,a_n)=A^{-1}([v_1,v_{n+1}],\dots,[v_n,v_{n+1}])^T,$$ hence $F(v_1,\dots,v_{n+1})$ is equal to a term in $L_\theta$.
Summarising, we have:
\begin{fact}
Put $L_\theta^F:=L_\theta\cup \{F_n:n\in\omega \}$ and let $T_0$ be a completion of the theory of fields admitting quantifier elimination in a language $L_K$.
Then, for every $m\in \omega$, the theories $_ST^{T_0}_m$ 
and $_AT^{T_0}_m$
 have quantifier elimination in $L^F_\theta\cup L_K$.

In particular,
 for every $m\in\omega\cup \{\omega\}$ (with $m$ even in the alternating case) the theories $_ST_m$ and $_AT_m$ have quantifier elimination in $L^F_\theta$,  $_ST^{RCF}_m$ has quantifier elimination in $L_\theta\cup \{\leq\}$ (by the discussion above) and $_AT^{RCF}_m$ has quantifier elimination in  $L^F_\theta\cup \{\leq\}$,  where $\leq$ is a binary relation symbol interpreted as the unique field ordering on $K$ in the real closed case. 
\end{fact}

The following fact follows from the proof of
\cite[Corollary 9.2.9]{Gr}: although in the case $T_m^*={_ST^{RCF}_m}$ it does not formally follow from \cite[Corollary 9.2.9]{Gr} as in a real closed field not all elements have square roots, this condition is only used to transform a normal basis to an orthonormal basis (see the proof of \cite[Proposition 9.1.5]{Gr}), which  clearly can be done over any real closed field if $[,]$ is positive definite.
\begin{fact}
 For any $m\in \omega\cup \{\infty\}$ the $L$-theory $T^*_m$ is complete.
\end{fact}

For a set or a tuple $A$, by $V(A)$ we mean the set of vectors belonging to $A$, and $\acl(A)$ [respectively, $\dcl(A)$] denotes the model-theoretic algebraic [definable] closure of $A$, that is, the set of elements  whose type over $A$ has finitely many realisations [only one realisation]. 
The following fact easily follows from quantifier elimination (cf. \cite[Proposition 9.5.1, Proposition 12.4.1]{Gr}).
\begin{fact}\label{dcl}
 Let $M=(V,K)\models T^*_\infty$ and $ A\subseteq M$. Then:\\
 (1) For any $v\in V\backslash \Lin_K(A)$ the type $\tp(v/A)$ is implied by $p_{v,A}(x):= \{ [x,a]=[v,a]:a\in V(A)\}\cup
 \{[x,x]=[v,v]\}\cup \{ \theta_n(a_1,\dots,a_n)\to\theta_{n+1}(a_1,\dots,a_{n},x):a_1,\dots,a_n\in V(A)\}$.\\
 (2) $\acl(A)\subseteq Lin_K(V(A))$.
 \end{fact}
 \begin{proof}
(1): Suppose $w\models p_{(v,A)}$. Then $v,w\notin \Lin_K(V(A))$, so there is a $K$-linear isomorphism $g:\Lin_K(V(A)\cup \{v\})\to
\Lin_K(V(A)\cup \{w\})$  fixing $\Lin_K(V(A))$ pointwise and sending $v$ to $w$. Then $g$ preserves $[,]$, so $g\cup \id_K$ is an elementary map by quantifier elimination. In particular, $\tp(w/A)=\tp(v/A)$. 

   (2): By finite character of $\acl$  we may assume that $A$ is finite.  Now, if $v\notin \Lin_K(V(A))$ then for any $u\in V$ which is orthogonal to $V(A)\cup \{v\}$ we have by (1) that  $\tp(v+u/A)=\tp(v/A)$, so in particular $\tp(v/A)$ has infinitely many realisations, i.e. $v\notin \acl(A)$.
 \end{proof}

\section{Dimension on $V$}\label{sec_1}
This section is in a large part a review of the results from \cite[Subsection 12.4]{Gr}, where the notions of dimension and codimension of a definable subset of the vector sort $V$ in $T_\infty$ were introduced. However, the definition of codimension there uses a false claim (see Remark \ref{comment} below), so we provide an argument fixing it.

In the rest of this paper,  $T^*_\infty$ means either $T_\infty $, in which case we put $\rk=\RM$, or $T^{RCF}_\infty $, in which case we put $\rk=\dim_t$ (see Definition \ref{dimt} below). 
When we write $X\subseteq V$ we mean that $X$ is a set of single elements of the sort $V$, but when we write $X\subseteq M$ where $M$ is a model [or when we say that $X$ is definable in $M$], we mean that $X$ is a [definable] set of arbitrary finite compatible tuples in $M$. We will be working in a fixed $\aleph_0$-saturated model $\C\models T^*_\infty$, which means every type in a single variable over a finite subset of $M$ is realised in $M$.  By Fact \ref{dcl} it is easy to see that this is equivalent to saying that the field of scalars $K(\C)$ has infinite transcendence degree over its prime subfield (we will need  $\aleph_0$-saturation only to choose generic elements in the proof of Theorem \ref{thm_gps}).

As in \cite{Gr}, {\bf we deal with the case of a symmetric bilinear form unless stated otherwise}, and the alternating case can be treated analogously by replacing an orthonormal basis by a symplectic basis. We will occasionally point out the main differences between the symmetric and the alternating case. In fact, the alternating case tends to be easier, as the condition $[x,x]=[v,v]$ in the type $p_{v,A}(x)$ implying $\tp(v/A)$ (see Fact \ref{dcl}(1)) is trivially satisfied by any vector $x$, so it can be omitted.

The following definition was introduced (in a more general version) in \cite[Section 12.1]{Gr}.
\begin{definition}
 If $M=(K(M),V(M))\models T^*_\infty$ and $V(M)$ is countably dimensional over $K(M)$, then an \emph{approximating sequence} for $M$ is a sequence $(N_r)_{r\in\omega}$ of substructures of $M$ with $K(N_r)=K(M)$ such that $N_r\models T_r$,  $M=\bigcup_{r\in \omega}N_r$, and $N_r\subseteq N_r'$ for all $r\leq r'$.
 
 In the alternating case, an approximating sequence
 is a sequence $(N_r)_{r\in\{2,4,\dots\}}$ satisfying analogous properties. 
 
 We will write $M=\bigcup^a_{r}N_r$ to mean that $(N_r)_{r}$ is an approximating sequence for $M$ (so in particular, $M\models T^*_\infty$).
\end{definition}

The following fact follows by, for example, the proofs of Theorem 1 and Corollary 1 in  \cite[Chapter 2.2]{Go}.
\begin{fact}\label{gross}
 If $M\models {_ST^*_\infty}$ and $V(M)$ has dimension $\aleph_0$ over $K(M)$, then $M$ has an approximating sequence $(N_r)_{r\in\omega}$, and for any such sequence we can find by the Gram-Schmidt process an orthonormal basis $(e_i)_{i\in \{1,2,\dots\}}$ for $V(M)$  over $K(M)$ such that $V(N_r)=\Lin_{K(M)}(e_1,\dots,e_r)$ for each $r\in \omega$. 
Similarly, in the alternating case, if $V(M)$ is countably dimensional over $K(M)$ then we can find an approximating sequence $(N_r)_{r\in \{2,4,\dots\}}$ for $M$ and  a symplectic basis $(e_i,f_i)_{i\in \omega}$ for $V(M)$ over $K(M)$ such that  $T_{2r}=\Lin_{K(M)}(e_1,f_1,\dots,e_r,f_r)$ for every $r\in \omega$. In both cases,  given an orthonormal [symplectic] basis $B$ for some $N_r$ with $r\in \omega$ [$r\in \{2l:l\in \omega\}$], we can find such an orthonormal [symplectic] basis for $M$ (or for any $N_{r'}$ with $r'\geq r$) which extends $B$.

Moreover, both in the symmetric and the alternating case, if $v_1,\dots,v_m\in V(M)$, then there is a $K(M)$-linear subspace $V_0$ of $V(M)$ such that $v_1,\dots,v_m\in V_0$ and $(K(M),V_0)\models T_{2m}$, and there is an approximating sequence $(N_r)_r$ for $M$ with $N_{2m}=(K(M),V_0)$. 
\end{fact}

By (the proof of) \cite[Lemma 10.1.3]{Gr} and quantifier elimination we have:
\begin{fact}\label{fact_biint}
Let $r\in \omega\cup \{\infty\}$ and $N=(V,K)\models T^*_r$.\\
(i) If $r\in \omega$, then the structure $N$ is definable (over some parameters) in the pure field $(K,+,\cdot)$.\\
(ii) For any $n\in \omega$, all definable [$\emptyset$-definable] in $N$ subsets of $K^n$ are definable [$\emptyset$-definable] in the pure field $(K,+,\cdot)$. 
\end{fact}

\begin{definition}\label{dimt}
 Let $N=(V,K)\models T^{RCF}_m$ for some $m\in \omega$. For any set $X$ definable in $N$ we put $\dim_t(X):=\dim^K_t(f[X])$ where $f$ is any definable bijection between $X$ and a subset of $K^n$ for some $n$ (note that $f[X]$ is definable in $(K,+,\cdot)$ by Fact \ref{fact_biint}(i)).  This does not depend on the choice of $f$, because for any other definable bijection $g$ between $X$ and a subset of $K^{m'}$, the sets $f[X]$ and $g[X]$ are in a $K$-definable bijection by Fact \ref{fact_biint}(ii).
\end{definition}

The following was stated in \cite[Corollary 12.4.2]{Gr} for definable subsets of $V$ and $T^*_\infty=T_\infty$, but exactly the same proof works for definable subsets of any $V^k$ and $T^*_\infty\in \{T_\infty,T^{RCF}_\infty\}$
using quantifier elimination and definability of $\rk$ (Fact \ref{fact_rm}(5)).
\begin{remark}\label{well_def_k}
If $M=\bigcup^a_r N_r$,  $M'=\bigcup^a_r N'_r$, $R\in \omega$ and  $X$ is a set definable over $N_R\cap N'_R$, then $\rk_{N_r}(X\cap N_r)=\rk_{N'_r}(X\cap N'_r)$
for all $r\geq R$. 
\end{remark}

\begin{remark}\label{biint}
 If $X$ is a set definable in $T^*_\infty$ over a model $M=\bigcup^a_r N_r$  and $X(M)\subseteq N_R$ for some $R\in \omega$, then for any $r\geq R$ we have $$\rk_{N_r}(X\cap N_r)=\rk_{N_R}(X\cap N_R).$$ 
 If $*=$ACF$_p$, then also $\rk_{N_R}(X\cap N_R)=\RM_{M}(X(M))$.
\end{remark}
\begin{proof}
If $*=$ACF$_p$, then, as the definable subsets of $X(M)=X\cap N_r=X\cap N_R$ in the sense of $N_R$, $N_r$ and $M$ all coincide by Fact \ref{fact_biint}, we  get $\rk_{N_r}(X\cap N_r)=\rk_{N_R}(X\cap N_R)=\rk_{M}(X(M)).$

If $*=$RCF, then the equality $\rk_{N_R}(X\cap N_R)=\rk_{N_R}(X\cap N_r)=\rk_{N_r}(X\cap N_r)$ follows directly from Definition \ref{dimt}, as an $N_R$-definable bijection between $X\cap N_r$ and a set definable in $K$ is in particular $N_r$-definable.
\end{proof}

For a tuple of parameters [tuple of single variables, respectively] $a$, by $l(a)$ we will mean the number of vectors [vector variables] in $a$.

By Fact \ref{dcl}(1) we have:
\begin{fact}\label{fact_dim0}
 Let $M$ be a countably dimensional model of $_ST^*_\infty$ with an orthonormal basis $(e_i)_{i\in \{1,2,\dots\}}$ and put $N_r=(K(M),\Lin_{K(M)}(e_1,\dots,e_r))$ for every $r\in\omega$. Suppose $R\in \omega$ and $a\in V(M)\backslash V(N_R)$. If we put $\beta_i=[a,e_i]$ for $i=1,\dots,e_R$ and $\gamma=[a,a]$, then $tp(a/N_R)$ is isolated by the formula $$\phi_{a,R}(x):=\bigwedge_{i=1,\dots,R}[x,e_i]=\beta_i\wedge [x,x]=\gamma $$
 (note that the formula saying that $x\notin V(N_R)$ can be omitted here, as $\gamma\neq \sum_{i=1,\dots,R}\beta_i^2$ since $a\notin N_R$, so $\phi_{a,R}(x)$ implies  $x\notin V(N_R)$).
\end{fact}

\begin{remark}\label{remark_1}
 Suppose $n\in \omega$ and $M=(K,V_0)$ is a model of $ _ST^*_n$ (i.e. a model of $_ST^{ACF_p}_n$ or of $_ST_n^{RCF}$). Let $c\in K\backslash \{0\}$ and assume $c>0$ in the real closed case. Then 
 $$\rk_{M}(\{v\in V_0:[v,v]=c\})=n-1,$$ and if $M\models{_ST^{ACF_p}_m}$, then we we also have $\DM(\{v\in V_0:[v,v]=c\})=1$.
\end{remark}
\begin{proof} 
Choose an orthonormal basis $(e_1,\dots,e_n)$ of $V_0$ over $K$. Then $ \Sigma x_ie_i\mapsto (x_1,\dots,x_n)$ gives a definable bijection between $\{v\in V_0:[v,v]=c\}$ and the sphere $\{(x_1,\dots,x_n)\in K^n: x_1^2+\dots +x_n^2=c\}$, which is well known to be an irreducible algebraic variety of dimension $n-1$, hence 
 $\rk(\{v\in V_0:[v,v]=c\})=n-1$ and, if $M\models{_ST^{ACF_p}_m}$, then $\DM(\{v\in V_0:[v,v]=c\})=1$.
\end{proof} 
\begin{corollary}\label{cor_1}
 With the notation of Fact \ref{fact_dim0}, we have
  $\rk_{N_r}(\phi_{a,R}(N_r))=r-R-1$ and, if $*=$ACF$_p$, then  $\DM_{N_r}(\phi_{a,R}(N_r))=1$ for any $r>R$ with $a\in N_r$.
\end{corollary}
\begin{proof}
Put $V_0:=\Lin_{K(M)}(e_{R+1},\dots,e_{r})$. Then clearly $(K(M),V_0)\models {_ST^*_{r-R}}$. Let $a_0$ be the projection of $a$ on $V(N_R)$. Then $w\mapsto w-a_0$ gives a definable bijection between $\phi_{a,R}(N_r)$
and $\{v\in V_0: [v,v]=[a,a]-[a_0,a_0]\}$. Hence the conclusion follows by Remark \ref{remark_1} (note $[a,a]-[a_0,a_0]\neq 0$ as $a\notin V(N_R)$). 
\end{proof}

 \begin{proposition}\label{correction}
 Suppose $M=\bigcup^a_{r\in \omega} N_r$ and  $X\subseteq V$ is a set definable by a formula $\phi(x,a)$ which is not contained in any finite-dimensional subspace of $V$. Let $R\in \omega$ be minimal such that $R\geq 4l(a)+1$ and $a\subseteq N_R$. Then there is $d\leq 2l(a)+1$ such that for any $r\geq R$ we have $$\rk_{N_r}(X\cap N_r)=r-d.$$
By Fact \ref{well_def_k}, $d$ does not depend on the choice of $M$ and $(N_r)_{r\in \omega}$.

Moreover, if $*=$ACF$_p$ and $r>R$ for $R$ as above, then  $\DM(X\cap N_r)=\DM(X\cap N_{R+1})$.
 \end{proposition}
\begin{proof}
By modifying $N_r$'s for $r<R$ (using Fact \ref{gross}), 
we may assume that $a\subseteq N_{2l(a)}$. Choose an orthonormal basis $(e_0,e_1,\dots)$ for $M$ such that $N_r=\Lin_{K(M)}(e_1,\dots,e_r)$ for every $r\in\omega$. 
\begin{claim}
There is $d\in\{0,1,\dots,2l(a)+1\}$ such that for any $r\geq 2l(a)+1$ we have $\rk_{N_r}((X\backslash N_{2l(a)}) \cap N_r)=r-d$. 
\end{claim}
\bpfc 
For any $v\in (X\backslash N_{2l(a)}) \cap N_r$ there is at least one, and  at most two vectors $w$ in $\phi_{v,N_{2l(a)}}(N_{2l(a)+1})$ (as defined in Fact \ref{fact_dim0}). Namely, if $v=v_0+v_1$ where $v_0\in N_{2l(a)}$ and $v_1$ is orthogonal to  $N_{2l(a)}$, then $w$ must be of the form $v_0+v_1'$ where $v_1'\in \Lin_{K(M)}(e_{2l(a)+1})$ and $[v_1',v_1']=[v_1,v_1]$. Clearly there are two possibilities on such a $v_1'$ in case $v_1\neq 0$, and they are additive  inverses of each other, and one such vector if $v_1=0$. Thus we have a definable surjection $$f^r_{2l(a)+1}:(X\backslash N_{2l(a)}) \cap N_r\to (X\backslash N_{2l(a)}) \cap N_{2l(a)+1}/\sim$$ sending $v\in (X\backslash N_{2l(a)}) \cap N_r$ to the at most two-element set $\phi_{v,N_{2l(a)}}(N_{2l(a)+1})$, where $\sim$ is the relation identifying $v_0+v_1$ with $v_0-v_1$ for $v_0\in N_{2l(a)}$ and $v_1\in \Lin_{K(M)}(e_{2l(a)+1})$. 
Put $$t:=\rk_{N_{2l(a)+1}}(\im(f^r_{2l(a)+1}))=\rk_{N_{2l(a)+1}}((X\backslash N_{2l(a)})\cap N_{2l(a)+1}/\sim)$$ 
($\sim$ above actually does not change the rank by Fact \ref{fact_rm}(4), as all $\sim$-classes are finite, and hence of rank $0$).
Clearly $t\leq 2l(a)+1$. Now, for any $w\in (X\backslash N_{2l(a)})\cap N_{2l(a)+1}$ we have that $(f^r_{2l(a)+1})^{-1}([w]_\sim)=\phi_{w,N_{2l(a)}}(N_r)$, which, by Corollary \ref{cor_1}, has rank $r-2l(a)-1$. Hence, by Fact \ref{fact_rm}(4) we get that $$\rk_{N_r}((X\backslash N_{2l(a)}) \cap N_r)=r-2l(a)-1+t=r-d$$ for $d:=2l(a)+1-t$. As $t$ did not depend on $r$, neither does $d$, so we are done.
\epf
Now, as $X\cap N_r=((X\backslash N_{2l(a)})\cap N_r)\cup (X\cap N_{2l(a)})$ and $\rk(X\cap N_{2l(a)})\leq 2l(a)\leq r-d$ for $r\geq 4l(a)+1$, we conclude by Fact \ref{fact_rm}(2) that $\rk_{N_r}(X\cap N_r)=r-d$ for every $r\geq 4l(a)+1$.

If $*=$ACF$_p$ and $r>R$ then we also have  $\RM((X\backslash N_{2l(a)})\cap N_r)>\RM(X\cap N_{2l(a)})$, which,  by Fact \ref{fact_rm}(7), implies that $$\DM(X\cap N_r)=
\DM((X\backslash N_{2l(a)})\cap N_r)=\DM((X\backslash N_{2l(a)})\cap N_{2l(a)+1}/\sim )=$$ $$=\DM((X\backslash N_{2l(a)})\cap N_{R+1})=\DM(X\cap N_{R+1})$$
where the second and third equalities follow by Fact \ref{fact_rm}(9) applied to $f^r_{2l(a)+1}$ and to $f^{R+1}_{2l(a)+1}$, respectively. 
\end{proof}

\begin{remark}\label{comment}
 Proposition 12.4.1 from \cite{Gr} uses the claim stated in the paragraph preceding it which says that  for $X$ and $(N_r)_{r\in \omega}$ as above (with $*=$ACF$_p$), then one has $\RM_{N_r}(X\cap N_r)\leq \RM_{N_{r+1}}(X\cap N_{r+1})+1$ for every $r$. This is not true even if we assume that $X$ is definable over $N_r$: for example, if $X=V\backslash \Lin_K(e_1,\dots,e_r)$, then $\RM_{N_r}(X\cap N_r)=\RM(\emptyset)=-1$, but $\RM_{N_{r+1}}(X\cap N_{r+1})=r+1$. 
\end{remark}

\begin{remark}\label{remark_alt}
If $[,]$ is alternating rather than symmetric, then in the setting of Proposition \ref{correction} we get that there is $d\leq 2l(a)$ such that for any $R\geq 2l(a)$ for which $a\subseteq N_{2R}$ we have $\rk_{N_{2R}}(X\cap N_{2R})=2R-d$ and if $*=ACF_p$ then $\DM_{N_{2r}}(X\cap N_{2r})=1$ for any $r>R$. The argument is very similar to that in the symmetric case: First, by Fact \ref{gross} we can find a substructure $N\models T^*_{2l(a)}$ of $N_{2R}$ containing $a$ with $K(N)=K(M)$, so we may assume that $a\subseteq 	N_{2l(a)}$. Next, we choose  $(e_i,f_i)_{i\in \omega}$  such that $(e_i,f_i)_{i\leq R}$ is a symplectic basis for $N_{2R}$ for every $R$ and let $\pi:N_{2R}\to N_{2l(a)}$ be  the projection with respect to the basis $(e_i,f_i)_{i}$. Then for any $R> l(a)$ we have  that $X\cap N_{2R}=(X\cap N_{2l(a)})\cup((\pi(X\cap N_{2R})\oplus\Lin_{K(M)}((e_i,f_i)_{l(a)<i\leq R}))\backslash N_{2l(a)})$  has rank $2R-2l(a)+\rk(\pi(X\cap N_{2R}))=2R-2l(a)+\rk(\pi(X\cap N_{2(l(a)+1)}))$, so we can put $d:=2l(a)-\rk(\pi(X\cap N_{2(l(a)+1)})$,  and the second assertion follows as in the symmetric case.
\end{remark}
Below we continue working with the symmetric case, the arguments in the alternating case being virtually the same.

Having Proposition \ref{correction}, the rest of the arguments from \cite[Subsection 12.4]{Gr} go through unchanged.

\begin{fact/def}\label{fact_dim}\cite[Proposition 12.4.1, Corollary 12.4.2, Definition 12.4.3]{Gr}
 Let $X\subseteq V$ be non-empty and definable in $T_{\infty}$ over a finite tuple $a$.  
 Then there exists $d\leq 2l(a)+1$ such that whenever $M=\bigcup^a_r N_r$, and $R\in \omega$ is such that $a\subseteq N_R$ and $R\geq 4l(a)+1$, then:
 $$\rk_{N_n}(X\cap N_n)=d\mbox{ for all $n\geq R$ or } \rk_{N_n}(X\cap N_n)=n-d\mbox{ for all $n\geq R$}.$$
 In the first case, we write $\Dim(X)=d$ and $\Codim(X)=\infty$, and in the second case we write $\Dim(X)=\infty$ and $\Codim(X)=d$.
In the first case $d$ can be chosen not greater than $2l(a)$.
\end{fact/def}


\begin{fact}\cite[Theorem 12.4.5]{Gr}\label{fact_def}
Let $X$ be a definable subset of the vector sort $V$. Then:\\
 (1) Exactly one of $\Dim(X)$ and $\Codim(X)$ is finite.\\
 (2) If $\phi(x,y)$ is a formula with $x$ a single variable, then  there are formulas without parameters $(\psi_n(y))_{n\in \omega}$ and $(\chi_n(y))_{n\in \omega}$ 
 such that, for each $n\in \omega$, one has $\Dim(\phi(x,b))=n\iff \models\psi_{n}(b)$ and  
$\Codim(\phi(x,b))=n\iff \models \chi_{n}(b)$.
\\
 (3) $\Dim(X)$ is finite iff $X$ is contained in a finite-dimensional subspace of $V$,  and in this case $\rk(X)=\Dim(X)$.
\end{fact}
\begin{remark}
 It is clear from the above result that there are formulas $\psi_{fin}(y)$ and 
 $\chi_{fin}(y)$ such that $\Dim(\phi(x,b))\in\omega\iff \models\psi_{fin}(b)$ and 
$\Codim(\phi(x,b))\in\omega\iff \models\chi_{fin}(b)$.  
\end{remark}

\section{Dimension on all definable sets}\label{sect_dim}

In this section, we define a notion of dimension of an arbitrary set definable in $T^*_\infty$ and we study its properties. On definable subsets of $V$ it is going to distinguish between infinite-dimensional sets of distinct codimensions, so formally it is not an extension of $\Dim$. Thus we are going to denote it by $\dim$ rather than $\Dim$ to avoid confusion.
We continue working in $T^*$ with $*\in \{$ACF$_p,$RCF$\}$.

Let   
$I=\{f\in (\mathbb{Z},+)^\omega:f(n)=0 \mbox{ for almost all } n\in \omega\}\leq (\mathbb{Z},+)^\omega$. Consider the quotient group:
 $$S:=(\mathbb{Z},+)^\omega/I.$$
For a function $f:\omega\to\mathbb{Z}$ we will write $[f]$ to mean $f/I$, and when $f$ is a given by a linear function over $\mathbb{Z}$, i.e. there are $d_0, d_1\in\mathbb{Z}$ such that $f(n)=d_0+d_1n$ for every $n\in \omega$, we shall identify $f$ with the linear polynomial $d_0+d_1n$ in variable $n$. For example, $[n]$ denotes the class of the function $g:\omega\to \mathbb{Z}$ given by $g(n)=n$ for any $n$.
Now put $$S_{\lin}:=\{[d_0+d_1n]: d_0,d_1\in \mathbb{Z} \}\leq S.$$
We will write $[f]\leq [g]$ if $f(k)\leq g(k)$ for almost all $k\in \omega$. For a partial function $f:\omega\nrightarrow \mathbb{Z}$ with domain co-finite in $\omega$, by $[f]$ we will mean $[\bar{f}]$ for any $\bar{f}:\omega\to \mathbb{Z}$ extending $f$.

\begin{remark}
 $(S_{\lin},+,\leq)$ is an ordered abelian group isomorphic to $(\mathbb{Z}\times\mathbb{Z},+,\leq_{lex})$.
\end{remark}
We will write $[f]<[g]$ when $[f]\leq [g]$ but $[f]\neq [g]$.

\begin{definition}
Suppose $X$ is a non-empty set definable in $T^*_\infty$ over a model $M=\bigcup^a_{r} N_r$.
Let $f_{X,M,(N_r)_{r\in \omega}}:\omega\to \mathbb{Z}$ be given by $f_{X,M,(N_r)_{r\in \omega}}=\rk_{N_r}(X\cap N_r)$ for each $r$.
Put 
$$dim(X):=[f_{X,M,(N_r)_{r\in \omega}}]\in S.$$
We also put $dim(\emptyset)=-1$.
\end{definition}

In the alternating case we define  $\dim(X)$ to be the class of the function $$f_{X,M,(N_r)_{r\in \{2,4,\dots\}}}:\{2,4,\dots\}\to \mathbb{Z}$$
with respect to being equal except finitely many points. However, we will see in  Theorem \ref{thm_dim} that the dimension of any definable set is given by a linear function (both in the symmetric and the alternating case), so, having Theorem \ref{thm_dim}, we can naturally identify $\dim(X)$ with an element of $S_{lin}$ also in the alternating case.
\begin{remark}
By Remark \ref{well_def_k}, if $X$ is also definable over $M'=\bigcup ^a_{r\in \omega} N_r'$, then 
 $[f_{X,M,(N_r)_{r\in \omega}}]=[f_{X,M',(N'_r)_{r\in \omega}}]$,
so the definition of $\dim(X)$ is independent of the choice of the model $M$ and the approximating sequence $(N_r)_{r\in\omega}$.
 \end{remark}

 We now aim to prove that the dimension of any set definable in $T^*_\infty$ belongs to $S_{lin}$ (so in particular the dimensions of the definable sets are linearly ordered). This will be proved first for definable subsets of $V^k$ by induction on $k$ simultaneously with some other statements.
 In particular, we define below a family of finite sets $D_{k,l}\subseteq S_{lin}$ which will turn out to contain the  dimension of any subset of $V^k$ definable over a set containing at most $l$ vectors.  
 
 
\begin{definition}\label{def_dkl}
For any $k,l\in \omega$ put $$D_{k,l}:=\{[d_0+d_1n]:0\leq d_1\leq k\mbox{ and }-d_1(2l+1)-k(k-1)\leq d_0\leq (k-d_1)2l+k(k-1)\}\subseteq S_{\lin}.$$ 
\end{definition}
The following property of the sets $D_{k,l}$ will be used in the inductive proof of Theorem \ref{thm_dim}. 
\begin{remark}\label{dkl}
 $D_{k,l}+D_{1,k+l}\subseteq D_{k+1,l}$ for any $l,k\in \omega$.
\end{remark}
\begin{proof}
 Suppose $[d_0+d_1n]\in D_{k,l}$ and  $[d_0'+d_1'n]\in D_{1,k+l}$. Then 
clearly $d_1+d_1'\leq k+1$ and $$-d_1(2l+1)-k(k-1)\leq d_0\leq (k-d_1)2l+k(k-1)$$ as well as
$$-d'_1(2k+2l+1)\leq d'_0\leq (1-d'_1)(2k+2l) $$
so $$-d_1(2l+1)-k(k-1)-d'_1(2k+2l+1)\leq d_0+d_0'\leq (k-d_1)2l+k(k-1)+(1-d'_1)(2k+2l),$$
which gives what we need, as 
$$-d_1(2l+1)-k(k-1)-d'_1(2k+2l+1)=
-(d_1+d'_1)(2l+1) -2d'_1k -k(k-1)\geq$$ $$\geq -(d_1+d'_1)(2l+1) -2k -k(k-1)=-(d_1+d'_1)(2l+1)-k(k+1)$$
and, similarly, on the right-hand side: $$ (k-d_1)2l+k(k-1)+(1-d'_1)(2k+2l)=
(k+1-d_1-d'_1)2l+k(k-1)+(1-d'_1)2k\leq$$ 
$$\leq
(k+1-d_1-d'_1)2l+k(k+1).$$
Hence $[d_0+d_1n]+[d'_0+d'_1n]=[d_0+d'_0+(d_1+d'_1)n]\in D_{k+1,l}$.
\end{proof}

 
 

\begin{definition}
Let $\alpha:\omega^2\to \omega$ be any
function such that:\\
\begin{itemize}
 \item $\alpha(k,l)\geq 2kl+2l+2k^2+1$ for any $k,l\in \omega$.
 \item $\alpha(k+m,l)\geq \alpha (k,l+m)$ 
 and $\alpha(k+m,l)\geq \alpha (k,l)$ for any $k,l,m\in\omega$.
\end{itemize}
Clearly, such a function can be constructed recursively on $k$.
 We will say that a definable set $X\subseteq V^k$ is \emph{nice}, if 
 $X=\emptyset$ or for each $a$ over which $X$ is definable one has 
 $\dim(X)=[d_0+d_1n]\in D_{k,l(a)}$, and whenever $M=\bigcup^a_{r\in \omega} N_r$, $R\geq \alpha(k,l(a))$, and $a\subseteq N_R$, then  we have $$\rk_{N_R}(X\cap N_R)=d_0+d_1R.$$
 \end{definition}
In the above situation, we know by the definition of $\dim$ that if $\dim(X)=[d_0+d_1n]$ then the equality $\rk_{N_R}(X\cap N_R)=d_0+d_1R$ holds for sufficiently large $R$, but the niceness property, saying that it holds for any $R$ with $R\geq \alpha(k,l(a))$ and   $a\subseteq N_R$, allows us to choose $R$ uniformly when we work with a uniformly definable family, which will be crucial in the proof of Lemma \ref{lemma_rk}(c) below. 

Note that by Fact \ref{fact_dim} we have that any definable  $X=\phi(\C,a)\subseteq V$ is nice: If $\Dim(X)=d_0\in \omega$, then $0\leq d_0\leq 2l(a)$ and  $\dim(X)=[d_0]$, so the inequalities $-d_1(2l(a)+1) \leq d_0\leq (k-d_1)2l(a)$ are satisfied as $d_1=0$ and $k=1$. If $\Codim(X)\in\omega$, then $\dim(X)=[d_0+n]$ for $d_0=-\Codim(X)$, so 
$d_1=1$ and $-2l(a)-1\leq d_0\leq 0$, so again the required inequalities hold. In both cases the equality $\rk_{N_R}(X\cap N_R)=d_0+d_1R$ holds for any $R\geq 4l(a)+1$ with $a\subseteq N_R$, hence for any $R\geq \alpha(1,l(a))$, as $\alpha(1,l(a))\geq 4l(a)+1$. 

We will eventually see in Theorem \ref{thm_dim} that all sets definable in $T^*_\infty$ are nice.

\begin{lemma}\label{lemma_geq}
 If $[d_0+d_1n],[d'_0+d'_1n]\in D_{k,l}$ and $[d_0+d_1n]> [d'_0+d'_1n]$, then $d_0+d_1r> d'_0+d'_1r $ for any $r\geq \alpha(k,l)$.
 \end{lemma}
\begin{proof}
If $d_1=d_1'$ then  $d_0> d_0'$, and the inequality is obvious, so assume $d_1>d_1'$. Then, by the inequalities in the definition of niceness we get: $$d_0+d_1r-(d_0'+d_1'r)=d_0-d_0'+(d_1-d_1')r\geq$$
$$\geq 
-d_1(2l+1)-k(k-1)-((k-d_1')2l+k(k-1))+(d_1-d_1')\alpha(k,l)=$$
$$=(d_1'-d_1)2l-2kl-d_1-2k(k-1)+(d_1-d_1')\alpha(k,l)=$$ $$=(d_1-d_1')(\alpha(k,l)-2l)-2kl-2k(k-1) -d_1\geq
\alpha(k,l)-2l-2kl-2k^2> 0,$$
so $d_0+d_1r> d_0'+d'_1r$.
\end{proof}

\begin{lemma}\label{last}
  If $M=\bigcup^a_r N_r$ and $\emptyset\neq X=\phi(M,a)$ for some formula $\phi(x;y)$, then $X\cap N_r\neq \emptyset$ for any $r\geq 2l(xy)$ such that $a\subseteq N_r$. 
 \end{lemma}
 \begin{proof} 
 This is similar to the proof of Fact \ref{dcl}(1).
 Let $c\in X$ and put $l:=l(a)$. We can find $e_1,\dots, e_{2l}, e'_{2l+1},\dots,e'_{r}$ such that $(e_1,\dots,e_r)$ and 
 $(e_1,\dots, e_{2l},e'_{2l+1},\dots,e'_{r})$ are orthonormal sequences, $V(a)\subseteq \Lin_{K(M)}(e_1,\dots, e_{2l})$, 
 $V(c)\subseteq \Lin_{K(M)}(e_1,\dots, e_{2l}, e'_{2l+1},\dots, e'_r)$ and $V(N_r)=\Lin_{K(M)}(e_1,\dots, e_r)$. Then letting $f=\id_{K(M)}\cup F$ where $F$ is a $K(M)$-linear function sending 
 $(e_1,\dots, e_{2l}, e'_{2l+1},\dots, e'_r)$ to $(e_1,\dots, e_r)$, we see by quantifier elimination that $\tp(f(c)/a)=\tp(c/a)$. In particular, $f(c)\in X\cap N_r$.
 \end{proof}

\begin{lemma}\label{lemma_rk}
a) If $X\subseteq Y$ then $dim(X)\leq dim(Y)$ \\
b)
 If $X_1,X_2\subseteq V^k$ are nice then 
$dim(X_1\cup X_2 )=\max(dim(X_1),dim(X_2))$.
If additionally $X_1$ and $X_2$ are definable over every tuple of parameters over which $X$ is definable, then $X_1\cup X_2$ is also nice.\\
c) 
Let $\pi:V^{k+m}\to V^k$ be the projection on the last $k$ coordinates (where $m\geq 1$). 
Suppose $X\subseteq V^{k+m}$ is definable and non-empty, all sections $X_y=\{x\in V^m:(x,y)\in X\}$ with $y\in \pi[X]$ are nice and they all have same dimension $s$, and  $\pi[X]$ is nice.
Then $\dim(X)=s+\dim(\pi[X])$.

If additionally $m=1$ then $X$ is nice.
\end{lemma}
\begin{proof}
 (a) Suppose $X=\phi(\C,a)$, $Y=\psi(\C,b)$, $M=\bigcup^a_{r\in\omega}N_r$ and 
  $a,b\subseteq N_R$ for some $R\in \omega$. Then for any $r\geq R$ we have $X\cap N_r\subseteq Y\cap N_r$, so $\rk_{N_r}(X\cap N_r)\leq \rk_{N_r}(Y\cap N_r)$ by Fact \ref{fact_rm}(1). 
 Hence $dim(X)\leq dim(Y)$.
\\
(b) 
Suppose $X_1=\phi(\C,a)$, $X_2=\phi(\C,b)$, $\dim[X_1]=[d_0+d_1n]$, $\dim{X_2}=[d_0'+d_1'n]$,   $M=\bigcup^a_{r\in\omega}N_r$, and  $R\geq \alpha(k,\max(l(a),l(b)))$ is such that $a,b\subseteq N_R$. We may assume $\dim(X_1)\geq \dim(X_2)$.
For any $r\geq R$ we have by Fact \ref{fact_rm}(2) that $$\rk_{N_r}((X_1\cup X_2)\cap N_r)=\rk_{N_r}((N_r\cap X_1)\cup (N_r\cap X_2))=\max (\rk_{N_r}(X_1\cap N_r),\rk_{N_r}(X_2\cap N_r)),$$
which equals $d_0+d_1r$ for almost all $r\in \omega$, and hence $\dim(X_1\cup X_2)=[d_0+d_1n]=\max(\dim(X_1),\dim(X_2))$.

Suppose additionally that $X_1$ and $ X_2$ are definable over any tuple of parameters over which $X_1\cup X_2$ is definable, and consider any $c$ such that $X_1\cup X_2$ (so also $X_1$ and $X_2$) is definable over $c$. Then the above remains true for any $r\geq \alpha(k,l(c))$ with $c\subseteq N_r$. For any such $r$, we know by niceness of $X_1$ and $X_2$ that $\rk_{N_r}(X_1\cap N_r)=d_0+d_1r$ and $\rk_{N_r}(X_2\cap N_r)=d_0'+d_1'r$. By Lemma \ref{lemma_geq} we have $d_0+d_1r\geq d'_0+d'_1r$, so 
$\rk_{N_R}((X_1\cup X_2)\cap N_r)=\max(d_0+d_1r,d'_0+d'_1r)=d_0+d_1r$, and hence $X_1\cup X_2$ is nice.
   \\
  c) 
  Assume $X=\phi(\C,a)$ and put $l=l(a)$.
 Let $d_0,d_1\in \omega$ be such that $s=[d_0+d_1n]$; as the sections of $X$ are nice, we have that $[d_0+d_1n]\in D_{m,k+l}$.
   Consider any $M=\bigcup^a_{r\in\omega}N_r$, and $r\geq \alpha(k+m,l)$ with $a\subseteq N_r$. 
  Put $Y=\pi[X]$. 
 
  For any $y\in Y\cap N_r$ we have $ (X\cap N_r)_y=X_y\cap N_r$, so, as $r\geq \alpha(k+m,l)\geq \alpha(m,l+k)=\alpha(m,l(ay))$ 
  and $X_y\subseteq V^m$ is a nice set definable over $ay$, we get  $$\rk_{N_r}((X\cap N_r)_y)=d_0+d_1r.$$ Note also that if $y\in Y= \pi[X]$ then  $X_y$ is a non-empty set definable over $ay$, so as, $r\geq \alpha(k+m,l)>2(k+m+l)$, it must meet $N_r$ by Lemma \ref{last}. Thus, $Y\cap N_r=\pi[X\cap N_r]$. 
  Hence, by Fact \ref{fact_rm}(4), we have $\rk_{N_r}(X\cap N_r)=\rk_{N_r}(Y\cap N_r)+d_0+d_1r$.
  As $Y$ is nice and $r\geq \alpha(k+m,l)\geq \alpha(k,l)$, we get that $\rk_{N_r}(Y\cap N_r)=d_0'+d_1'r$ for $d'_0,d'_1$ such that $\dim(Y)=[d_0'+d_1'n]\in D_{k,l}$. So $\rk_{N_r}(X\cap N_r)=d_0+d'_0+(d_1+d'_1)r$.
  Hence $\dim(X)=[d_0+d'_0+(d_1+d'_1)n]=s+\dim(Y)$.
 
 If, additionally, $m=1$, then  $\dim(X)\in D_{k,l}+D_{1,k+l}\subseteq D_{k+1,l}$ by Remark \ref{dkl}, so $X$ is nice.
  \end{proof}

\begin{theorem}\label{thm_dim}
We work in $T^*_\infty$.\\
(a) For any $k\in \omega$, every non-empty definable subset of $V^k$ is nice. In particular,   $\dim(X)\in D_{k,l(a)}$ for any $X\subseteq V^{k}$ definable over a finite tuple $a$. \\
(b) Suppose $k\in \omega$, $x=(x_1,\dots,x_k)$ where each $x_i$ is a variable of the sort $V$, and $y$ is an arbitrary tuple of variables. Then for any formula $\phi(x;y)$ over $\emptyset$ and any $s\in D_{k,l(y)}$ the set $$D_{\phi(x;y), s}:=\{a\in \C: \dim(\phi(x;a))=s\}$$ is $\emptyset$-definable.

\end{theorem}
\begin{proof}
We will prove (a) and (b) simultaneously by induction on $k$.

 When $k=1$, we know that  (a) and (b) both hold by Section \ref{sec_1}. 
 
 Suppose now $k\geq 1$ and  (a) and (b) are true for $1,2,\dots,k$. Consider
 any formula $\phi(x;y)$ over $\emptyset$ with $x=(x_1,\dots,x_{k+1})$, where each $x_i$ is a variable of the sort $V$. 
 
  Consider any $a\in \C$ compatible with $y$, and write $X_a=\phi(\C;a)\subseteq V^{k+1}$. For $b\in V^k$ put $$X_{b,a}=\phi(\C;b,a)\subseteq V .$$

 For any $s\in D_{1,k+l(y)}$ let $\chi_s(x_2,\dots,x_{k+1};y)$ be 
 a formula over $\emptyset$ such that $$\models \chi_s(v_2,\dots,v_{k+1};w) \iff \dim(\phi(\C;v_2,
 \dots,v_{k+1},w))=s \mbox{ for all }v_2,\dots,v_{k+1},w \in \C, $$ (such a formula exists, as (b) holds for $k=1$). Put $X_{s,a}=\{b\in V:\dim(X_{b,a})=s\}=\chi(\C;a)$.
 
 Then letting $\pi:V^{k+1}\to V^k$ be the projection on the last $k$ coordinates, we have for each  $s\in D_{1,k+l(y)}$ and each
 $b\in X_{s,a}$ that $\dim((\pi|_{X_a})^{-1}(b))=\dim(X_{b,a})=s$ and $X_{b,a}$ is nice by the inductive hypothesis, as is $X_{s,a}$. Thus, by Lemma \ref{lemma_rk}(c), we get that $(\pi|_{X_a})^{-1}[X_{s,a}]$  is nice. Now
 $$                                                                         
 X_a=\bigcup_{s\in D_{1,k+l(y)}}(\pi|_{X_a})^{-1}[X_{s,a}]                                                         $$
  and $(\pi|_{X_a})^{-1}[X_{s,a}]$ is nice for each $s\in D_{1,k+l(y)}$, so by Lemma \ref{lemma_rk}(b) we conclude that $X_a$ is nice, which proves part (a) of the theorem for $k+1$.
 
 Lemma
\ref{lemma_rk} gives us also that  $\dim (X_a)=\max_{s\in D_{1,k+l(y)}}\dim((\pi|_{X_a})^{-1}[X_{s,a}])$ and\\ $\dim((\pi|_{X_a})^{-1}[X_{s,a}])=s+\dim(X_{s,a})$. Hence, putting 
  $$I=D_{1,k+l(y)}\times D_{k,l(y)},$$ we get that for any $a\in \C$ compatible with $y$ we have $\dim(\phi(x_1,\dots,x_{k+1};a))\in \{s+t:(s,t)\in I\}$.
So fix any $(s_0,t_0)\in I$, and put $I_{=s_0+t_0}=\{(s,t)\in I: s+t=s_0+t_0\}$ and $I_{>s_0+t_0}=\{(s,t)\in I: s+t>s_0+t_0\}$. Then  $$\dim(\phi(x_1,\dots,x_{k+1};a))=s_0+t_0\iff$$ $$\iff (\bigvee_{(s,t)\in I_{=s_0+t_0}} dim(\chi_s(x_2,\dots,x_{k+1};a))=t)\wedge(\bigwedge_{(s,t)\in I_{>s_0+t_0}}
\neg dim(\chi_s(x_2,\dots,x_{k+1};a))=t).$$
By the inductive hypothesis, for any $(s,t)\in I$ the condition $\dim(\chi_s(x_2,\dots,x_{k+1};a))=t$
is definable (in the variable $a$), so, by the above equivalence, the condition $dim(\phi(x_1,\dots,x_{k+1};a))=s_0+t_0$  is definable as well. This proves that (b) holds for $k+1$.
\end{proof}

\begin{remark}\label{remark_bij}
 If there is a definable bijection $f$ between $X\subseteq V^{k}$ and $Y\subseteq V^{k'}$, then $\dim(X)=\dim(Y)$.
\end{remark}
\begin{proof}
 If $X$, $Y$ and $f$ are all definable over $N_R$, where  $M=\bigcup^a_r N_r$ and $R\in \omega$, then for any $r\geq R$ we have that $f[X\cap N_r]\subseteq Y\cap N_r$ as $\dcl(N_r)=N_r$ by Fact \ref{dcl}, so $X\cap N_r\subseteq f^{-1}[Y\cap N_r]$, and similarly $f^{-1}[Y\cap N_r]\subseteq X\cap N_r$. Hence $$\rk_{N_r}(X\cap N_r)=\rk_{N_r}(f^{-1}[Y\cap N_r])=\rk_{N_r}(Y\cap N_r)$$ by Fact \ref{fact_rm}(3), so 
 $\dim(X)=\dim(Y)$.
\end{proof}

We now extend the definition of dimension to all sets definable in $T_\infty$.
\begin{definition}\label{def_fk}
 If $X$ is any set definable in $T^*_\infty$, so $X\subseteq V^k\times K^{m}$ for some $k,m\in \omega$, then we let $\dim(X)=\dim(X')$, where $X'$ is any definable subset of $V^{k'}$ for some $k'\in \omega$ for which there is a definable (with parameters) bijection between $X$ and $X'$. Such an $X'$ always exists, as we have a definable injection $f_{k,(e_1,\dots,e_m)}:V^k\times K^m\to V^{k+1}$ given by $f_{k,(e_1,\dots,e_m)}(v,a_1,\dots, a_m)=(v,a_1e_1+\dots+a_me_m)$, where $(e_1,\dots,e_m)$ is any fixed linearly independent tuple of vectors from $V$.
 
 Moreover, $\dim(X)$ is well defined by Remark \ref{remark_bij} above.
\end{definition}

Now we summarise the properties of $\dim$ following from what we have proved so far.

\begin{corollary}\label{big_col}
We work in $T^*_\infty$\\
 (1) $\dim$ is $\emptyset$ definable.\\
 (2) If $X\subseteq Y$ are definable then $\dim(X)\leq \dim(Y)$. \\
 (3) $\dim(X_1\cup X_2)=\max(\dim(X_1),\dim(X_2))$ for any definable $X_1$ and $X_2$.\\
 (4) If  $f:X\to Y$ is a definable surjection such that $\dim(f^{-1}(y))=s$ for each $y\in Y$, then $\dim(X)=\dim(Y)+s$ unless $Y$ is empty.
 \end{corollary}
\begin{proof}
 (1) Consider any formula $\phi(x,y;z)$ where $x$ is a variable of the sort $V^k$ and $y$ is a variable of the sort $K^m$, and any $s\in S_{\lin}$. Let $x'$ be a variable of the sort $V^{k+1}$ and put $$\psi(x';z,e_1,\dots,e_m)=(x'\in \im (f_{k,(e_1,\dots,e_m)})\wedge \phi(f_{k,(e_1,\dots,e_m)}^{-1}(x');z)), $$ where $e_1,\dots,e_m$ are some linearly independent vectors from $V$. By Theorem \ref{thm_dim}(b) there is a formula $\chi_s(z,e_1,\dots,e_m)$ such that, for any $z$, $$\models \chi_s(z,e_1,\dots,e_m)\iff \dim(\psi(x';z,e_1,\dots,e_m))=s\iff \dim(\phi(x,y;z))=s.$$
 As this holds for any linearly independent vectors $e_1,\dots,e_m$, we may replace the formula $\chi_s(z,e_1,\dots,e_m)$ by the $L(\emptyset)$-formula
 $\exists_{v_1,\dots,v_m}(\theta_m(v_1,\dots,v_m)\wedge\chi_s(z,v_1,\dots,v_m))$\\
 (2) Suppose $X\subseteq Y\subseteq V^k\times K^m$. Then $f_{k,(e_1,\dots,e_m)}[X]\subseteq f_{k,(e_1,\dots,e_m)}[Y]$, so $$\dim(X)=\dim(f_{k,(e_1,\dots,e_m)}[X])\leq \dim(f_{k,(e_1,\dots,e_m)}[Y])=\dim(Y)
 $$ by Lemma \ref{lemma_rk}(a). \\
 (3) This follows by Lemma \ref{lemma_rk}(b) using the injection $f_{k,(e_1,\dots,e_m)}$ again.\\
 (4) As any subset of $V^k\times K^m$ is in a definable bijection with a subset of $V^{k'}\times K^{m'}$ for any $k'\geq k, m'\geq m$, we may assume (by modifying $X$, $Y$ and $f$) that there are $k$ and $m$ such that $X,Y\subseteq V^{k-1}\times K^m$. Then applying $f_{k-1,(e_1,\dots,e_m)}$ we may assume $X,Y\subseteq V^{k}$. Put $$Z:=\{(x,y)\in X\times Y : y=f(x)\},$$ and let $\pi_1: X\times Y\to X$ and $\pi_2:X\times Y\to Y$ be the projections. Note that $\dim(X)=\dim(Z)$ as $\pi_1|_{Z}:Z\to X$ is a definable  bijection. Moreover, for any $y\in Y$ we have $X_y=f^{-1}(y)$ has dimension $s$.  Thus, by Lemma \ref{lemma_rk}(c) we have  $\dim(X)=\dim(Z)=$\\$=\dim(\pi_2[Z])+s=\dim(Y)+s.$
 \end{proof}
 Note the above properties correspond to the main properties of Morley rank in strongly minimal theories (and of the topological dimension in RCF) listed in Fact \ref{fact_rm}. However, a major difference is that the set of values of $\dim$ is not well ordered. Nevertheless,  if we work with a fixed finite tuple of variables and a fixed finite tuple of parameters, the set of possible values of $\dim$ is finite.

\begin{remark}\label{rk_rm}
If $X\subseteq V^k$ is definable in $T^*_\infty	$ then $\dim(X)\in\omega $
if and only if $X\subseteq (V_0)^k$ for some finite dimensional $K(\C)$-linear subspace $V_0$ of $V$. 
Moreover, if these equivalent conditions hold and $*=$ACF$_p$ then $\dim(X)=[\RM(X)]$.
\end{remark}
\begin{proof}
 The implication from right to left follows from Remark \ref{biint}.
  
 Assume $\dim(X)=[d]\in\omega $. Then, for each $i\in\{1,\dots, k\}$, the projection $\pi_i(X)$ of $X$ on the $i$-th coordinate must have finite dimension (bounded by $d$), so, by Fact \ref{fact_def}(3), there is some finite dimensional $V_i\leq V$ such that $\pi(X)\subseteq V_i$. This means that $X\subseteq (\Sigma_i V_i)^k$ so we can put $V_0:=\Sigma_i V_i$.
 
 The `moreover' clause now follows by Remark \ref{biint} again.
\end{proof}

Finally, we define the dimension of a type. As the set 
of values of $\dim$ in $T^*_\infty$ is not well ordered, in general we need to use its Dedekind completion $\bar{S_{\lin}}$. 
\begin{definition}\label{def_dim_tp}
 Let $\pi(x)$ be a partial finitary type (i.e. $x$ is a finite tuple of variables) in $T^*_\infty$ over a set $A$. We put $$\dim(\pi(x)):=\inf_{\pi(x)\vdash \phi(x)\in L(\C)}\dim(\phi(x))=\inf_{\pi(x)\vdash \phi(x)\in L(A)}\dim(\phi(x))\in \bar{S_{\lin}}.$$
 Note that $\dim(\pi(x))\in S_{\lin}$ if  $A$ contains only finitely many vectors, as in this case the dimension of any formula in $x$ over $A$ belongs to the finite set $D_{l(x),l(A)}$.
\end{definition}

\begin{proposition}\label{rk_extension}
 Let $\pi(x)$ be a partial finitary type in $T^*_{\infty}$ over $A$. Then there exists $p(x)\in S(A)$ with $\dim(p(x))=\dim(\pi(x))$.
\end{proposition}
\begin{proof}
 Put $$p_0(x):=\pi(x)\cup \{\neg\phi(x)\in L(A): 
 \dim(\phi(x))<\dim(\pi(x))  \}.$$
 
 We claim that $p_0(x)$ is consistent. For if not, then there is a finite $\pi_0(x)\subseteq \pi(x)$ and formulas $\phi_1(x),\dots,\phi_n(x)$ such that  $\dim(\phi_i(x))<\dim(\pi(x))$ for every $i$ and $\bigwedge\pi_0(x)\wedge(\bigwedge_{1\leq i\leq n}\neg\phi_i(x))$ is inconsistent. 
Then $\models \bigwedge\pi_0(x)\rightarrow \bigvee_{1\leq i\leq n} \phi_i(x) $.
But, by Lemma \ref{lemma_rk}(b), $$\dim(\pi_0)=\dim((\bigwedge \pi_0)\wedge\bigvee_{1\leq i\leq n} \phi_i(x))=
\max_{1\leq i\leq n} \dim(\pi_0\wedge\phi_i(x))<\dim(\pi(x)),$$ which is a contradiction. 

Hence $p_0$ is consistent, and we can take $p$ to be any completion of $p_0$.
\end{proof}

\begin{notation}
 For $s,s'\in S_{lin}$ we will write: \begin{itemize}
    \item     $s\sim s'$ if $s-s'\in \{[d]:d\in \mathbb{Z}\}$, \item $s\lesssim s'$ if $s\sim s'$ or 
 $s\leq s'$, \item   $s\ll s'$ if $s\leq s'$ and $\neg (s\sim s')$.
  \end{itemize}
\end{notation}

\begin{definition}\label{def_dimgen}
 (1) We write $\dim(a/b)$ to mean $\dim(\tp(a/b))$. By the discussion in Definition \ref{def_dim_tp}, if $a$ and $b$ are finite, then $\dim(a/b)\in S_{lin}$.\\
 (2)
 If $X$ is a set (type)-definable over $a$ and $b\supseteq a$, then we say that an element $c\in X$ is \emph{generic} [quasi-generic] in $X$ over $b$ if $\dim(c/b)=\dim(X)$ [$\dim(c/b)\sim \dim(X)$]. 
 \end{definition}
 By Proposition \ref{rk_extension}, for any $X$ definable over $a$ and any $b\supseteq a$ there exists a generic in $X$ over $b$ (in some model of $T^*_\infty$ containing $b$).
 If $b$ is finite, then such a generic can be found in $\C$, as we are assuming that $\C$ is $\aleph_0$-saturated.

%

\section{Lascar's equality and the connection between $\dim$ and $\dim_{\Lin}$}\label{sect_lascar}
The following additivity property is an analogue of Lascar's equality, which holds, for example, for Morley rank in strongly minimal theories (and more generally, for Lascar $U$-rank assuming the the ranks in the statement are finite).

\begin{proposition}\label{lascar}(Lascar's equality for $\dim$)
 If $a,b,c\in \C$ are finite tuples, then $\dim(ab/c)=\dim(a/bc)+\dim(b/c)$.
\end{proposition}
\begin{proof}
 First, we will show that $\dim(ab/c)\geq\dim(a/bc)+\dim(b/c)$.
 Consider any formula $\phi(x,y)\in tp(ab/c)$. Then $\phi(x;b,c)\in tp(a/bc)$, so $s:=\dim(\phi(x;b,c))\geq \dim(a/bc)$. Now, by Corollary \ref{big_col}(1) there is a formula $\chi(y;c)$ over $c$ such that  $$ \models \chi(d;c) \iff \dim(\phi(x;d;c))=s$$ for any $d$ compatible with $y$.  Then $\chi(y;c)\in tp(b/c)$, so $t:=\dim(\chi(y;c))\geq \dim(tp(b/c))$. Now, by Corollary \ref{big_col}(4) applied to $\phi(x,y;c)\wedge \chi(y;c)$ and the projection on the $y$-coordinate, we get that $\dim(\phi(x,y;c))\geq \dim(\phi(x,y;c)\wedge \chi(y;c))=s+t\geq \dim(a/bc)+\dim(b/c)$. This shows that $\dim(ab/c)\geq\dim(a/bc)+\dim(b/c)$.
 
 Now, choose a formula $\psi(x;b,c)\in tp(a/bc)$ such that $s':=\dim(\phi(x;b,c))=\dim(a/bc)$. Again by \ref{big_col}(1), there is a formula $\chi(y;c)$ over $c$ such that  $$ \models \chi(d;c) \iff \dim(\psi(x;d,c)=s')$$ for any $d$ compatible with $y$. Clearly $ \chi(y;c)\in tp(b/c)$ so if we choose $\xi(y;c)\in tp(b/c)$ such that $t':=\dim(\xi(y;c))=\dim(b/c)$, then we also have $\chi(y;c)\wedge \xi(y;c)\in tp(b/c)$, hence $\dim(\chi(y;c)\wedge \xi(x;c))=t'$.
  Now applying Corollary \ref{big_col}(4) to the formula $$\delta(x,y;c):=\psi(x,y;c)\wedge \chi(y;c)\wedge \xi(y;c)$$ and the projection on the $y$-coordinate, we get $\dim(\delta(x,y;c))=s'+t'$.
  As $\delta(x,y;c)\in tp(ab/c)$, we conclude that $\dim(ab/c)\leq s'+t'=\dim(a/bc)+\dim(b/c)$.
\end{proof}

\begin{proposition}\label{lin_dim}
 If $a,b$ are finite tuples and $\dim(a/b)=[d_0+d_1n]$, then $d_1$ is equal to the linear dimension $\dim_{\Lin}(a/b)$ of $V(a)$ over $V(b)$, that is, the size of a maximal subset of $V(a)$ which is $K$-linearly independent over $\Lin_K(V(b))$.
\end{proposition}
\begin{proof}
Put $l:=\dim_{\Lin}(a/b)$ and let $(a_1,\dots,a_k)$ be all vectors in  $a$, and let $(c_1,\dots,c_m)$ be all scalars in $a$. We may assume $(a_1,\dots,a_l)$ is a maximal $K$-linearly independent over $\Lin_K(V(b))$ subtuple of $a$. Write $V(b)=\{b_1,\dots,b_p\}.$
Let $\phi(x_1,\dots,x_k)$ be a formula over $b$ expressing that $x_{l+1},\dots,x_{k}\in \Lin_K(V(b),x_1,\dots,x_l)$, and let $f:\phi(\C)\times K^m\to V^l\times K^m\times K^{(k-l)(l+p)}$ be a map sending a tuple 
$(x_1,\dots,x_k,y_1,\dots,y_m)$ to $(x_1,\dots,x_l,y_1,\dots,y_m,A)$ where $A$ is an $ (l+p)\times (k-l)$-matrix such that $A(x_1,\dots,x_l,b_1,\dots,b_p)^T=(x_{l+1},\dots,x_k)$. Then $f$ is a $b$-definable injection of $\phi(\C)\times K^m$ into $V^l\times K^{m+(k-l)(l+p)}$.
As $a\models \phi(x_1,\dots,x_k)$, we get  $$\dim(a/b)\leq \dim(\phi(\C)\times K^m)\leq 
\dim(V^l\times K^{m+(k-l)(l+p)})=[m+(k-l)(l+b)+ln].$$
This shows that $d_1\leq l=\dim_{\Lin}(a/b)$.

It is left to prove that $d_1\geq \dim_{\Lin}(a/b)$, which we do  by induction on $\dim_{\Lin}(a/b)$. If $\dim_{\Lin}(a/b)=1$ then $\dim(a/b)\geq \dim(a_1/b)$ and $a_1\notin \Lin_K(V(b))$ so $\dim (a_1/b)$ is infinite by Fact \ref{fact_def}(3), i.e. $d_1\geq 1$. Now for the inductive step use Fact \ref{fact_def}(3) together with Lascar's equality.
\end{proof}

together with  Lascar's equality are key ingredients in the proof of Theorem \ref{thm_gps}.

\begin{corollary}\label{cor_int}
 For any finite tuples $a,b,c$ we have  $\dim(a/b)\sim \dim(a/bc)$ if and only if $\Lin_K(V(ab))\cap \Lin_K(V(bc))=\Lin_K(V(b))$.
\end{corollary}
\begin{proof}
 If $\Lin_K(V(ab))\cap \Lin_K(V(bc))=\Lin_K(V(b))$ then any  tuple $(a_1,\dots,a_d)$ of elements of $V(a)$ which is $K$-linearly independent over $\Lin_K(V(b))$ is also linearly independent over $\Lin_K(V(bc))$, so $\dim(a/b)\sim \dim(a/bc)$ by Proposition \ref{lin_dim}.
 
 Conversely, assume $\dim(a/b)\sim \dim(a/bc)$ and let $(a_1,\dots,a_d)$ be a maximal tuple of elements of $V(a)$ which is $K$-linearly independent over $V(bc)$. By  Proposition \ref{lin_dim} and the assumption, $(a_1,\dots,a_d)$ is also maximal $K$-linearly independent over $V(b)$. Hence, any element of $\Lin_K(V(ab))\cap \Lin_K(V(bc))$ is of the form $\Sigma_{i\leq d}\alpha_i a_i+b_1=c_1$ for some $\alpha_i\in K$, $b_1\in \Lin_K(V(b))$, and $c_1\in  \Lin_K(V(bc))$, so 
 $\Sigma_{i\leq d}\alpha_i a_i=c_1-b_1\in \Lin_K(V(bc))$. As $(a_1,\dots,a_d)$ is linearly independent over $\Lin_KV(bc)$, we get that $\Sigma_{i\leq d}\alpha_i a_i=0$ and $\Sigma_{i\leq d}\alpha_i a_i+b_1=b_1\in \Lin_K(V(b))$. Thus, $\Lin_K(V(ab))\cap \Lin_K(V(bc))=\Lin_K(V(b))$.
\end{proof}

\section{Finiteness of multiplicity and its consequences}\label{sect_mlt}

In this section we will define multiplicity of a set definable in $T_\infty$ in analogy with Morley degree and we will prove that the multiplicity of any set definable in $T_\infty$ is finite.
We will deduce that any group interpretable in $T_\infty$ which has finite Morley rank is definable in $T_\infty$, and hence is an algebraic group over $K$ (we will also prove an analogous result for $T^{RCF}_\infty$), as well as some other consequences of finiteness of multiplicity, including definability of generic types in $T_\infty$.

\begin{definition}
Let $X$ be definable in $T_\infty$. We let the multiplicity of $X$, written $\Mlt(X)$, be the maximal number $m\in \omega$ such that there are pairwise disjoint definable sets $X_1,\dots,X_m$ with $X_i\subseteq X$ and $\dim(X_i)=\dim(X)$ for each $i\in \{1,\dots,m\}$ if such a number $m$ exists, and $\infty$ otherwise. 
\end{definition}

\begin{proposition}\label{prop_mlt}
We work in $T_\infty$.\\
 (1) If $X\subseteq Y$ and $\dim(X)=\dim(Y)$ then $\Mlt(X)\leq \Mlt(Y)$.\\
 (2) If $\dim(X_1)<\dim(X_2)$, then $\Mlt(X_1\cup X_2)=\Mlt(X_2)$.\\
 (3) If $\dim(X_1)=\dim(X_2)=s$ then $Mtl(X_1\cup X_2)\leq \Mlt(X_1)+\Mlt(X_2)$, and equality holds when $\dim(X_1\cap X_2)<s$.\\
 (4) If $f:X\to Y$ is a definable function such that $\Mlt(Y)=m\in \omega$ and there are $s\in S_{\lin}$ and $m'\in \omega$ such that $\dim(f^{-1}(y))=s$ and $\Mlt(f^{-1}(y))\leq m'$ for every $y\in Y$, then $\Mlt(X)\leq mm'$.\\
 (5) If $\dim(X)\in \omega$ then $\Mlt(X)=\DM(X)$.
\end{proposition}
\begin{proof}
 (1),(2), and (3) follow easily from the definition of $\Mlt$ and the properties of $\dim$ (Corollary \ref{big_col}) and (5) follows from Proposition \ref{rk_rm}. 
 
 Let us prove (4).
  Let $Y_1,\dots,Y_{m}$ be sets partitioning $Y$ with $\dim(Y_i)=\dim(Y)$. By (3) applied to the sets $f^{-1}[Y_1],\dots,f^{-1}[Y_m]$ we may assume that $m=1$ and $Y_1=Y$. 
 Suppose for a contradiction that there are pairwise disjoint 
  $X_1,\dots,X_{m'+1}\subseteq X$ with $\dim(X_i)=\dim(X)$.
 For each $i\in \{1,\dots,m'+1\}$ put $Z_i:=\{y\in Y:\dim(f^{-1}(y)\cap X_i)=s\}\subseteq Y$. Then each $Z_i$ is definable by Corollary \ref{big_col}(1) and $\dim(Z_i)=\dim(Y)$ by Corollary \ref{big_col}(4) applied to $f$ and to $f|_{X_i}$. As $\Mlt(Y)=1$, using induction and (2) we easily get that $\dim(\bigcap_{i\in\{1,\dots,m'+1\}}Z_i)=\dim(Y)$. In particular, there exists $y\in \bigcup_{i\in\{1,\dots,m'+1\}}X_i$ and we have that $(f^{-1}[y]\cap Z_i)_{i\in \{1,\dots,m'+1\}}$ are pairwise disjoint subsets of $f^{-1}(y)$ of dimension $s$, a contradiction to $\Mlt(f^{-1}(y))=m'$.
\end{proof}

\begin{theorem}\label{thm_mlt}
We work in $T_\infty$.\\
 (1) For every formula $\phi(x;y)$ over $\emptyset$ there exists $m_{\phi(x;y)}\in \omega$ such that for every $R\in \omega$ and  avery $N\models T_R$ containing $a$ we have 
 $\DM_{N}(\phi(\C,a)\cap N)\leq m_{\phi(x;y)}$.\\
 (2) Every formula in $T_\infty$ has finite multiplicity. 
\end{theorem}
\begin{proof}
Using the functions $f_{k,e_1,\dots,e_m}$ (see Definition \ref{def_fk}), we may assume that $x$ is a tuple of $k$ vector variables for some $k\in \omega$. We will now prove the statement by induction on $k$. 

For any fixed $R_0\in\omega$, by quantifier elimination in $T_{R_0}$ and Fact \ref{fact_rm}(6) we easily get a bound on
$\DM_{N}(\phi(\C;a)\cap N)$ with $a\subseteq N\models T_{R_0}$ depending only on $\phi(x;y)$ and on $R_0$.   Also, we know by the proof of Proposition \ref{correction} that if $\dim(\phi(\C;a))\in \omega$ then $\phi(M,a)\subseteq N_{2l(a)}$ for some $M=\bigcup^a_rN_r$ with $a\subseteq N_{2l(a)}$. Hence we may restrict ourselves to considering only $R\geq \alpha(l(x),l(y))$ and $a$ such that $\dim(\phi(\C;a))\notin \omega$.

First, assume $k=1$ so $x$ is a single vector variable. 
By Proposition \ref{correction}, if $R\geq \alpha(1,l(y))$ (so $R> 4l(a)+1$),  $a\in N\models T_R$ and $\Dim(\phi(\C;a))\notin \omega$, then  $\DM_{N}(\phi(\C;a)\cap N)$  is equal to $\DM_{N'}(\phi(\C;a)\cap N')$ for any $N'\models T_{4l(a)+2}$ containing $a$ with $K(N)=K(M)$ and $N'\subseteq N$. This, in turn, is bounded independently from $a$ and $N$ by quantifier elimination  in $T_{4l(a)+2}$ and Fact \ref{fact_rm}(6), which completes the proof when $k=1$. 

Now, assume that $k\geq 1$ and  we have numbers $m_{\psi(x_1,\dots,x_i;y)}$ satisfying the assertion for  each $\phi(x_1,\dots,x_i;y)$ with $x_1,\dots,x_i$ being single variables of the sort $V$ and $i\leq k$. Consider any $\phi(x;y)$ with $x=(x_1,\dots,x_{k+1})$, where each $x_i$ is a single variable of the sort $V$. 
 As every definable set is nice, there are at most $D:=|D_{1,k+l(y)}|$ possibilities $s_1,\dots,s_D\in D_{1,k+l(y)}$ on $\rk(\phi(x_1;v_2,\dots,v_{k+1},w))$
for $v_2,\dots,v_{k+1},w\in \C$ with $\phi(x_1;v_2,\dots,v_{k+1},w)\neq \emptyset$. 
For each $i\leq D$ let $\chi_{s_i}(x_2,\dots,x_{k+1},y)$ be a formula over $\emptyset$ such that $$\models \chi_{s_i}(v_2,\dots,v_{k+1},w)\iff \dim(\phi(x_1;v_2,\dots,v_{k+1},w))=s_i.$$
Put $$m_{\phi(x_1,\dots,x_{k+1};y)}:=\Sigma_{i\leq D}(m_{\phi(x_1;x_2,\dots,x_{k+1},y)}m_{\chi_{s_i}(x_2,\dots,x_{k+1};y)})$$
(the numbers on the right-hand side are already defined by the inductive hypothesis).
Consider any $R\geq \alpha(k+1,l(y))$, $N\models T_R$, and $a\subseteq N$ compatible with $y$.
As $\alpha (k+1,l(y))\geq \alpha(1,k+l(y))$, for every $i\leq D$ there is $t_i\in \omega$ such that for every $v_2,\dots,v_{k+1},w\in N$, if $\models \chi_{s_i}(v_2,\dots,v_{k+1},w)$ then $\RM_{N}(\phi(\C;v_2,\dots,v_{k+1},w)\cap N)=t_i$.

Let $\pi:V^{k+1}\to V^k$ be the projection on the last $k$ coordinates. 
Put $X=\phi(\C;a)$ and $X_{s_i}:=X\cap \pi^{-1}[\chi_{s_i}(\C;a)]$ for each $i\leq D$.

 By Lemma \ref{last} we get (as in the proof of Lemma \ref{lemma_rk}(c)) that $\pi[X_{s_i}\cap N]=\pi[X_{s_i}]\cap N=\chi_{s_i}(\C;a)\cap N$.
Also, for any $v_2,\dots,v_{k+1}\in \pi[(X_{s_i}\cap N)]$ we know  that $$\DM((\pi|_{X_{s_i}\cap N})^{-1}(v_2,\dots,v_{k+1}))=\DM( N\cap  {\pi|_{X_{s_i}}}^{-1}(v_2,\dots,v_{k+1}))=$$ $$=\DM(N\cap\phi(\C;v_2,\dots,v_{k+1},a)) \leq m_{\phi(x_1;x_2,\dots,x_{k+1},y)}.$$
Thus, by Fact \ref{fact_rm}(8) and by Fact \ref{fact_rm}(9) applied to the functions
$\pi|_{X_{s_i}\cap N}$ we have $$\DM_{N}(X\cap N)\leq \Sigma_{i\leq D}\DM_N(X_{s_i}\cap N)\leq \Sigma_{i\leq D}(m_{\phi(x_1;x_2,\dots,x_{k+1},y)}\DM_{N}(\chi_{s_i}(\C;a)\cap N))\leq$$ $$\leq \Sigma_{i\leq D}(m_{\phi(x_1;x_2,\dots,x_{k+1},y)}m_{\chi_{s_i}(x_2,\dots,x_{k+1};y)}),$$ 
as required. This completes the induction.
\\
(2) Choose $M=\bigcup^a_r N_r$ containing $a$. Let $m:=m_{\phi(x,y)}$ be the number given by (1). We claim that $\Mlt(\phi(x,a))\leq m$. If not, then there exist pairwise disjoint sets $X_1,\dots,X_{m+1}\subseteq \phi(M,a)=:X$ definable in $M$ over some finite $b\subseteq M$. Let $R\geq \alpha(l(x),l(ab))$ be such that each $X_i$ is definable over $N_R$ and $ab\subseteq N_R$.
Then, as $X$ and all $X_i$'s are nice, we have  $\RM_{N_R}(X_i\cap N_R)=\RM_{N_R}(X\cap N_R)$ for every $i\leq m+1$, so $\DM_{N_R}(X\cap N_R)>m$, which contradicts the choice of $m$.
 \end{proof}

\begin{corollary}\label{cor_elim}
 If $G$ is a group definable in $T^*_\infty$ and $H\triangleleft G$ is a definable normal subgroup such that $\dim(G)-\dim(H)\in \omega$, then $G/H$ is definably isomorphic to a definable in $T^*_\infty$ group of finite dimension.
\end{corollary}
\begin{proof}
Put $d:=\dim(G)-\dim(H)\in\omega$.
 Let $a$ be a finite tuple over which $H$ and $G$ are definable in a variable $x$, and choose  $M=\bigcup^a_{r}N_r$ and $R\geq \alpha(l(x),l(a))$ such that $a\subseteq N_R$. Then, by niceness of $H$ and $G$, we have $\rk_{N_r}(G\cap N_r)=\rk_{N_r}(H\cap N_r)+d$ so $\rk_{N_r}((G\cap N_r)/(H\cap N_r))=d$ for every $r\geq R$. Note that $H\cap N_r\triangleleft G\cap N_r$ for every $r\geq R$ as each $N_r$ is $\dcl$-closed.
 \begin{claim}\label{cl_sep}
 There is $r_0\geq R$ such that for every $r\geq r_0$ the  definable embedding of groups
  $$h_{r_0,r}:(G\cap N_{r_0})/(H\cap N_{r_0})\to (G\cap N_{r})/(H\cap N_{r})$$ given by $g(H\cap N_{r_0})\mapsto g(H\cap N_{r})$ is surjective. 
 \end{claim}
\bpfc
 {\bf Case 1}: $*=$ACF$_p$.\\ 
 By Theorem \ref{thm_mlt} we know there is $m$ such that $\DM_{N_r}(G\cap N_r)\leq m$ for every $r\in \omega$, so also $\DM_{N_r}((G\cap N_r)/(H\cap N_r))\leq m$ for any $r\geq R$ by Fact \ref{fact_rm}(9). If $h_{r,r+1}$ is not surjective for some  $r\geq R$, then $h_{r,r+1}[(G\cap N_r)/(H\cap N_{r})]=(G\cap N_r)/(H\cap N_{r+1})$ is a proper subgroup of the group $(G\cap N_{r+1})/(H\cap N_{r+1})$ of the same dimension $d$, so 
 $$\DM_{N_r}((G\cap N_r)/(H\cap N_r))=\DM_{N_{r+1}}((G\cap N_r)/(H\cap N_{r+1}))<$$ $$<\DM_{N_{r+1}}((G\cap N_{r+1})/(H\cap N_{r+1})),$$ so, by boundedness of $\DM_{N_r}((G\cap N_r)/(H\cap N_r))$ (by $m$) there is $r_0\geq R$ such that for every $r\geq r_0$ the embedding $h_{r,r+1}$ is surjective, and so is $h_{r_0,r}=h_{r-1,r}h_{r-2,r-1}\dots h_{r_0,r_0+1}$. \\
 {\bf Case 2}: $*=$RCF.\\
 We claim that $r_0:=R$ works.
 For any $r\geq r_0$ we have that $\dim_t((G\cap N_r)/(H\cap N_r))=d=\dim_t((G\cap N_R)/(H\cap N_R))=\dim_t((G\cap N_R)/(H\cap N_r))$, so, by Fact \ref{fact_rm}(0) and (4), the index $$[(G\cap N_r)/(H\cap N_r):(G\cap N_R)/(H\cap N_R)]= [G\cap N_r:(G\cap N_R)\cdot (H\cap N_r)]$$ is finite. Note that the group $G\cap N_R$ normalises $H\cap N_r$ so 
 $G_0:=(G\cap N_R)\cdot (H\cap N_r)=\{x\cdot y:x\in G\cap N_R, y\in H\cap N_r\}$ is a definable subgroup of $G\cap N_r$. Now for any $g\in G\cap N_r$ the  coset $g G_0\in (G\cap N_r)/G_0$ is algebraic in $N_r$ over $N_R$. As RCF eliminates imaginaries and algebraic closure coincides with definable closure in RCF, this implies that the coset $g\cdot G_0$ is definable over $N_R$, hence also over $N_{r_0}$. Also, $g\cdot G_0$ is definable  over $a,g$, so, as $r_0\geq 4l(x)+2l(a)$, we get by Lemma \ref{last} that $g\cdot G_0\cap N_{r_0}\neq \emptyset$.
 This shows that $h_{r_0,r}$ is surjective, which completes the proof of the claim.
 \epf
 By the claim, for every $g\in G(M)$ there is $g'\in 
 G\cap N_{r_0}$ with $gH=g'H$. As $M\prec \C$, we must also have that for every $g\in G(\C)$ there is $g'\in G\cap \Lin_{K(\C)}(N_{r_0})$ with $g/H=g'/H$, so $$G/H=(G\cap \Lin_{K(\C)}(N_{r_0}))/(H\cap   \Lin_{K(\C)}( N_{r_0}))$$	 is definable in $K(\C)$ (by elimination of imaginaries in $K$), and hence it is definable in $\C$ and has finite dimension.
\end{proof}

\begin{remark}\label{rm_2}
 If $X$ is definable in $T_\infty$ and $E$ is a definable equivalence relation on $X$ such that $\RM(X/E)<\omega$ (in $\C$ expanded by the sort $X/E$ and the quotient map $X\to X/E$), then for every $s\in S_{lin}$ for which there is $x\in X$ with $\dim(x_E)=s$ we have $ \dim(X_s) -s\in \omega$, where $X_s=\{x\in X:\dim(x_E)=s\}$.
\end{remark}
\begin{proof}
 Put $l:=\RM(X/E)$, and let $M=\bigcup^a_{r\in \omega}N_r$ with $X$ and $E$ definable over some finite $b\subseteq N_R$ for some $R\geq \alpha(\l(x),l(x)+l(b))$, where $x$ is a variable in which $X$ is definable. If $\dim(X_s)=[d_0+d_1n]$ and $s=[d_0'+d_1'n]$ with $d_1'<d_1$, then for every $r\geq R$  we have $\RM_{N_r}(X_s\cap N_r)>l+d_0'+d'_1r$ by Lemma \ref{lemma_geq},  and $\dim(x_E)\cap N_r=d_0'+d_1'r$ for each $x\in X\cap N_r$ by niceness of the $xb$-definable set $x_E$. But, as $\RM_{N_r}(X_s\cap N_r/E)\leq \RM(X/E)=l$, we get by Fact \ref{fact_rm}(4) applied to the quotient map $X_s\cap N_r\to (X_s\cap N_r)/E$  that $$\RM_{N_r}(X_s\cap N_r)=d_0'+d_1'r+\RM_{N_r}((X_s\cap N_r)/E)\leq l+d_0'+d_1r,$$  (note $(X_s\cap N_r)/E$ is definable in $N_r$ by elimination of imaginaries in ACF$_p$). This is a contradiction.
\end{proof}

\begin{corollary}\label{tfae}
Let  $G$  be a group definable in $T^*_\infty$ and let $H\triangleleft G$ be a definable normal subgroup. Then, the following are equivalent:\\
(1) $\dim(G)-\dim(H)\in \omega$.\\
(2) $G/H$ is definably isomorphic to a definable in $T^*_\infty$ group of finite dimension (hence to an algebraic group over $K$ if $*=$ACF$_p$ and to a semialgebraic group over $K$ if $*=RCF$).\\
If $*=$ACF$_p$, then these conditions are also equivalent to:\\
(3) $\RM(G/H)<\omega$.
\end{corollary}
\begin{proof}
(1) implies (2) by Corollary \ref{cor_elim}. 
(2) implies (3) by Remark \ref{rk_rm} and it  implies (1) as well by Corollary \ref{big_col}(4) applied to the quotient map $G\to G/H$ (where we identify $G/H$ with a definable group definably isomorphic to it).
(3) implies (1) by Remark \ref{rm_2} applied to the equivalence relation $E$ on $G$ given by: $E(g,g')\iff gH=g'H$. 
 \end{proof}
From finiteness of multiplicity in $T_\infty$, we also conclude definability of generic types.

\begin{proposition}\label{cor_mlt}
(1) Let $X$ be definable in $T_\infty$ over a model $M$. Put $m=\Mlt(X)$. Then there are exactly $m$ complete generic types in $X$ over $M$.\\
(2) Let $M\models T_\infty$ and let $p(x)\in S(M)$ be such that $\dim(p(x))\in S_{lin}$. Then $p(x)$ is definable. Hence, each generic type in every definable set is definable.
\end{proposition}
\begin{proof}
 (1) Suppose first that there are $m+1$ distinct generics $p_1,\dots,p_{m+1}\in S(M)$ in $X$. Let $\phi(x)$ be a formula over $M$ defining the set $X$. Choose pairwise inconsistent formulas $\phi_i(x)\in p_i$ for $i\leq m+1$. Then, as $\phi_i(x)\wedge \phi(x)\in p_i$ for each $i$, we must have $\dim(\phi_i(x)\wedge \phi(x))=\dim(X)$ as each $p_i$ is generic in $X$. This shows that $\Mlt(X)\geq m$, a contradiction.
 
On the other hand, by definability of $\dim$ (Corollary \ref{big_col}(1)) we can find pairwise disjoint $X_1,\dots,X_m\subseteq X$ definable over $M$ with $\dim(X_i)=X$ for each $i$, and choose a generic $p_i\in X_i$ for each $i$. Then $p_i$'s are pairwise distinct generics in $X$.\\
 (2) As $\dim(p(x))\in S_{\lin}$, we can choose $\phi(x)\in p(x)$ such that $\dim(\phi(x))=\dim(p(x))$. By definability of $\dim$ there are formulas $\phi_1(x),\dots,\phi_m(x)$ over $M$ of dimension $\dim(\phi(x))$ which partition $\phi(x)$, and one of them must belong to $p(x)$. So we may assume $\Mlt(\phi(x))=1$. Now consider any formula $\psi(x;y)$. Then for any $a\subseteq M$ compatible with $y$  we have that $\psi(x,a)\in p(x)$ iff $\dim(\psi(x;a)\wedge \phi(x))=\dim(\phi(x))$: If $\psi(x,a)\in p(x)$ then $\psi(x;a)\wedge \phi(x)\in p(x)$ so $\dim(\psi(x;a)\wedge \phi(x))=\dim(\phi(x))$; conversely, if the latter holds, then the generic type in  $\psi(x;a)\wedge \phi(x)$ over $M$ is also generic in $\phi(x)$, so is equal to $p(x)$ by (1), as $\Mlt(\phi(x))=1$. Thus $\psi(x;a)\in p(x)$. 
 
 As the condition $\dim(\psi(x;a)\wedge \phi(x))=\dim(\phi(x))$ is definable by Corollary \ref{big_col}(1), we get that $p(x)$ is a definable type. 
\end{proof}

\section{Definable groups and fields}\label{sec_gps}
In this section we will prove our main results about groups and fields definable in $T^*_\infty$. Let us start with some examples. Clearly, any algebraic group over the field of scalars $K$ is definable in $T_\infty$ and any semialgebraic group over $K$ is definable in $T^{RCF}_\infty$. Another class of examples is obtained from the natural actions of linear algebraic groups over $K$ on Cartesian powers of the (infinite-dimensional) vector space $V$:
\begin{example}\label{examp}
 Let $M=(V,K)$ be a model of $T^*_\infty$ and $k\in \omega$.\\
 (1) Suppose $H\leq GL_k(K)$ is a linear algebraic group. Consider the semidirect product $G:=V^k\rtimes H$, where the action of $H$ on $V^k$ is induced by scalar multiplication.
Then $G$ is definable in $M$ in a natural way, with its universe being a definable subset of  $V^k\times K^{k^2}$ consisting of pairs $(v,A)$ with $v\in V^k$ and $A\in H$. \\
(2) Let $(G,\cdot)$ be the Heisenberg group of $[,]$, that is, $G=V\times V\times K$ and $(v,w,a)\cdot (v',w',a')=(v+v',w+w', a+a'+[v,w'])$ for $(v,w,a), (v',w',a')\in G$. Then $(G,\cdot)$ is  definable in $M$ (in an obvious way).
\end{example}

We say a definable group $G$ is \emph{connected} if it has no definable subgroup of finite index.
\begin{remark}
 Every group definable in $T_\infty$ has a connected component, that is, a definable connected subgroup of finite index.
\end{remark}
\begin{proof}
 Let $M$ be a model over which $G$ is definable. By Proposition \ref{cor_mlt} there are only finitely many generic types in $G$ in $S(M)$. Let $p_1,\dots,p_m$ be all of them. Then for any $i\leq m$  and $g\in G(M)$ we have that $g\cdot p_i(x):=
 \{\phi(g^{-1}\cdot x):\phi(x)\in p_i(x)\}\in S(M)$ is also a generic in $G$, so $G$ acts naturally on $p_1,\dots,p_m$. Let $G_0$ be the kernel of this action. Now, if we choose pairwise inconsistent $\phi_i(x)\in p_i(x)$ of dimension $\dim(G)$ and multiplicity $1$, then $G_0=\{g\in G: \bigwedge_i \dim(\phi_i(x)\wedge \phi_i(g^{-1}\cdot x))=\dim(G)\}$ (cf. the proof of Proposition \ref{cor_mlt}), so $G_0$ is definable by Corollary \ref{big_col}. As $[G:G_0]<\omega$, we must have $\dim(G_0)=\dim(G)$. Now only one of the types $p_1,\dots,p_m$ contains the formula `$x\in G_0$', as otherwise we would have some $g_i\in G_0\cap \phi_i(M)$ and $g_j\in G_0\cap \phi_j(M)$ for $i\neq j$, so $g_ig_j^{-1}\cdot p_j=p_i$, a contradiction, as $g_ig_j^{-1}\in G_0$.
 Hence $G_0$ has only one generic type, and so $\Mlt(G_0)=1$ by Proposition \ref{cor_mlt}. This clearly implies that $G_0$ is connected.
\end{proof}
By a [semi] algebraic group in our context we mean a group interpretable in $T^*_\infty$ which is definably isomorphic to a [semi] algebraic group over the field of scalars $K$. Thus, for example, although the group $(V,+)$ might be abstractly isomorphic to the group $(K,+)$ in a particular model $(K,V)\models T_\infty$, it is not an algebraic group in our sense, as there is no definable bijection between $V$ and any set definable in $K$. Accordingly, we say that a definable group $G$ is ([semi] algebraic-by-abelian)-by-[semi] algebraic, if there are definable $N\triangleleft G$ and $N_0\triangleleft N$ such that $N_0$ and $G/N_0$ are [semi] algebraic and $N/N_0$ is abelian.

Let $g,h\in G$ where $G$ is a group.  We will usually write the product of $g$ and $h$ as $gh$ omitting the multiplication symbol. To avoid confusion with a pair, below we will use commas in tuples.
By $g^h$ we mean the conjugate $hgh^{-1}$ of $g$ by $h$, and by $[g,h]$ we mean the commutator $ghg^{-1}h^{-1}$ of $g$ and $h$. By $G'$ we denote the commutator subgroup of $G$, that is, the subgroup of $G$ generated by the set $\{[x,y]:x,y\in G\}$. \begin{theorem}\label{thm_gps}
 Let $G$ be a group definable in $T^*_\infty$. Then $G$ is (algebraic-by-abelian)-by-algebraic when $*=$ACF$_p$ and
 (semialgebraic-by-abelian)-by-semialgebraic when $*=$RCF.
\end{theorem}
\begin{proof}
Let $G$ be a group definable in $T_\infty$ [or in $T^{RCF}_\infty]$ over some finite tuple $a$. 
We may assume that $G\subseteq V^k$ for some $k\in \omega$
and that $a$ is a subtuple of any element of $G$. 

Put $$N:=\{x\in G:\dim(C_G(x))\sim \dim(G)\}.$$

\begin{claim}\label{cl2}
 $N\triangleleft G$ and $N$ is $a$-definable.
\end{claim}
\bpfc
First, we show that $N$ is a subgroup of $G$. Take any $g_1,g_2\in G$.
Let $M=\bigcup^a_{r\in\omega}N_r$ and $R\in \omega$ be such that  $g_1,g_2,a\subseteq  N_R$. Consider any $r\geq R$. Note that $C_G(g_1)\cap N_r=C_{G\cap N_r}(g_1)$ and $C_G(g_2)\cap N_r=C_{G\cap N_r}(g_2)$ are both subgroups of the group $G\cap N_r\leq G$, as $N_r$ is definably closed by Fact \ref{dcl}(2). Hence $(G\cap N_r)/(C_G(g_1)\cap C_G(g_2)\cap N_r)$ embeds $N_r$-definably into $((G\cap N_r)/(C_G(g_1)\cap N_r))\times ((G\cap N_r)/(C_G(g_2)\cap N_r))$  by $$g(C_G(g_1)\cap C_G(g_2)\cap N_r)\mapsto (g(C_G(g_1)\cap N_r),g(C_G(g_2)\cap N_r)).$$ 
So $$\rk_{N_r}((G\cap N_r)/(C_G(g_1)\cap C_G(g_2)\cap N_r))\leq$$ $$\leq \rk_{N_r}((G\cap N_r)/(C_G(g_1)\cap N_r))+\rk_{N_r}((G\cap N_r)/(C_G(g_2)\cap N_r)) .$$ By Fact \ref{fact_rm}(4) applied to the corresponding quotient maps, this means that $$\rk_{N_r}(G\cap N_r)-\rk_{N_r}(C_G(g_1)\cap C_G(g_2)\cap N_r)\leq$$ $$\leq \rk_{N_r}(G\cap N_r)-\rk_{N_r}(C_G(g_1)\cap N_r)+\rk_{N_r}(G\cap N_r)-\rk_{N_r}(C_G(g_2)\cap N_r).$$ As this holds for any $r\geq R$, we get that $$\dim(G)-\dim(C_G(g_1)\cap C_G(g_2))\leq \dim(G)-\dim(C_G(g_1))+\dim(G)-\dim(C_G(g_2))\in \omega.$$

As $C_G(g_1\cdot g_2)\supseteq  C_G(g_1)\cap C_G(g_2)$, we conclude that $\dim(C_G(g_1\cdot g_2))\sim\dim(G)$, so $g_1\cdot g_2\in N$ and $N$ is a subgroup of $G$. Also, for any $g\in G$ and $h\in N$   we have that $\dim (C_G(h^g))=\dim((C_G(h))^g)=\dim(C_G(h))$, as $(C_G(h))^g$ and $C_G(h)$ are in a definable bijection. This shows that $N$ is normal in $G$. Finally, $N$ is $a$-definable by Corollary \ref{big_col}(1).
 \epf

Let $h$ be a generic in $G$ over $a$ and let $g$ be a generic in $G$ over $a,h$.
Write $h=(w_1,\dots,w_k)$ and $g=w_{k+1},\dots,w_{2k}$ (where $w_i\in V$ for each $i\leq 2k$). Let $j_1,\dots,j_l\in \{1,\dots, 2k\}$ be such that $w_{j_1},\dots,w_{j_l}$ is a basis of $W:=\Lin_{K}(w_1,\dots,w_{2k})$ over $K$.

For any $x=(v_1,\dots,v_k)\in G$ there are $i_1,\dots,i_m\in\{1,\dots,k\}$ such that $(v_{i_1},\dots,v_{i_m}, w_{j_1},\dots,w_{j_l})$ is a basis of $\Lin(W,v_1,\dots,v_k)$. As this is expressible by a formula $\phi(x)$ with parameters $h, g$ and there are only finitely many possibilities on the tuple $(i_1,\dots,i_m)\in\{1,\dots,k\}^m$ (with $m\leq k$), by Corollary \ref{big_col}(3) there must be some such tuple for which the set 
$$X:=\{x=(v_1,\dots,v_k)\in G: (v_{i_1},\dots,v_{i_m}, w_{j_1},\dots,w_{j_l})\mbox{ is a basis of }\Lin_K(W,v_1,\dots,v_k)\}$$ is generic in $G$.
We may assume $(i_1,\dots,i_m)=(1,\dots,m)$.
Notice that for any $x=(v_1,\dots, v_k)\in X$ we have $gx^h\in dcl(x,g,h)\subseteq \Lin_K(w_{i_1},\dots,w_{i_l},v_1,\dots,v_m)$, so we can define a function let $f:X\to K^{k(l+m)}$  such that for every $x=(v_1,\dots,v_k)\in X$  $$\mbox{if }f(x)=Y=(Y_1,Y_2)\mbox{ with } Y_1\in M_{l\times k}(K),Y_2\in M_{m\times k}(K)\mbox{ then }$$ $$gx^h=Y_1\cdot (w_{i_1},\dots,w_{i_l})^T+Y_2\cdot (v_1,\dots,v_m)^T.$$ 
As $f$ is a definable function and $\dim(\im(f))\leq [k(l+m)]$, by Corollary \ref{big_col}(4) there must be 
some $C=(C_1,C_2)\in K^{k(l+m)}$ such that $$\dim(f^{-1}(C))\sim \dim(G).$$
Then  for $x=(v_1,\dots,v_k)\in f^{-1}(C)$ we have 
$$ gx^h=C_1\cdot (w_{i_1},\dots,w_{i_l})+C_2\cdot (v_1,\dots,v_m).$$

By Lemma \ref{rk_extension} we can choose $g_1\in f^{-1}(C)$ such that $\dim(g_1/C,h,g,a)=\dim(f^{-1}(C))\sim\dim(G)$, and $g_2\in f^{-1}(C)$ such that 
$\dim(g_2/C,h,g,g_1,a)\sim\dim(G)$. Write $g_1=(v_1,\dots,v_k)$ and $g_2=(v'_1,\dots,v'_k)$. So $$gg_1^h=C_1\cdot (w_{i_1},\dots,w_{i_l})^T+C_2\cdot (v_1,\dots,v_m)^T\mbox{ and }
gg_2^h=C_1\cdot (w_{i_1},\dots,w_{i_l})^T+C_2\cdot (v'_1,\dots,v'_m)^T.$$
So $t:=C_1\cdot (w_{i_1},\dots,w_{i_l})^T=gg_1^h-C_2\cdot (v_1,\dots,v_m)^T\in dcl(g_1,gg_1^h,C,a)$, hence $$gg_2^h=t+C_2\cdot (v'_1,\dots,v'_m)^T\in dcl(g_1,gg_1^h,g_2,C,a).$$
Thus, $$(g_1^{-1}g_2)^h=(gg_1^h)^{-1}gg_2^h\in dcl(g_1,gg_1^h,g_2,C,a).$$

So, choosing $h_1$ to be a generic in $\tp(h/g_1,gg_1^h,g_2,C,a)$ over $h,g_1,gg_1^h,g_2,C,a$, we get 
$$(g_1^{-1}g_2)^h=(g_1^{-1}g_2)^{h_1},$$ hence $$g_1^{-1}g_2\in \C_G(hh_1^{-1}).$$ 
\begin{claim}\label{cl3}
 $\dim(hh_1^{-1}/g_1^{-1}g_2,a)\sim \dim(G)$.
\end{claim}
\bpfc
By Lascar's equality (Proposition \ref{lascar}) we have \begin{align} \dim(h/C,a)\geq \dim(h/a)-\dim(C/a,h) \sim\dim(G) \end{align}
as $h$ is generic in $G$ over $a$ and $\dim(C/a,h)\in \omega$. Also
\begin{align}
 \dim(g_1/h,C,a)\sim \dim(G) 
\end{align}
by the choice of $g_1$. Now, as $g_1$ is quasi-generic over $g,h,a$ we have that $$\dim(g_1,g,h/a)=\dim(g_1/g,h,a)+\dim(g/h,a)+\dim(h/a)\sim 3\dim(G)$$ by Lascar's equality, but also $\dim(g_1,g,h/a)=\dim(g/g_1,h,a)+\dim(g_1,h/a)$, so  $\dim(g/g_1,h,a)\sim \dim(G)$ as $\dim(g_1,h/a)\leq 2\dim(G)$. Hence 
\begin{align}
\dim(gg_1^h/g_1,h,C,a)\sim \dim(gg_1^h/g_1,h,a)=\dim(g/g_1,h,a)\sim \dim(G),
\end{align}
where the equality follows by invariance of $\dim$ under definable bijections. We also have \begin{align}
                                                                                   \dim(g_2/gg_1^h,g_1,h,C,a)\sim \dim(G)                                                                                       \end{align}
by the choice of $g_2$. Now, by (1),(2),(3),(4), and Lascar's equality we have $\dim(g_2,gg_1^h,g_1,h/C,a)\sim 4\dim(G)$. But $\dim(g_2,gg_1^h,g_1,h/C,a)=\dim(h/g_2,gg_1^h,g_1,C,a)+\dim(g_2,gg_1^h,g_1,C,a)$, so $$\dim(h/g_2,gg_1^h,g_1,C,a)\sim \dim(G)$$
as $\dim(g_2,gg_1^h,g_1,C,a)\lesssim 3\dim(G)$. As $h_1$ is generic in $\tp(h/g_1,gg_1^h,g_2,C,a)$ over $(h,g_1,gg_1^h,g_2,C,a)$, it follows that  $$\dim(hh_1^{-1}/h,g_1,gg_1^h,g_2,C,a)=\dim(h_1/h,g_1,gg_1^h,g_2,C,a) \sim \dim(G), $$ so also
$\dim(hh_1^{-1}/g_1^{-1}g_2,a)\sim\dim(G)$
which completes the proof of the claim.
\epf 

As $hh_1^{-1}\in \C_G(g_1^{-1}g_2)$ and   $\dim(hh_1^{-1}/g_1^{-1}g_2,a)\sim \dim(G)$ by Claim \ref{cl3}, we get $\dim C_G (g_1^{-1}g_2)\sim \dim(G)$. This shows that $g_1^{-1}g_2\in N$, so, as $\dim(g_1^{-1}g_2/a)\sim \dim(G)$ and $N$ is definable over $a$, we conclude that $\dim(N)\sim \dim(G)$,
so $G/N$ is an algebraic [semialgebraic] group by Corollary \ref{tfae}.

It is left to show that $N$ is algebraic-by-abelian [semialgebraic-by-abelian].
For any $x\in N$ we have that all fibers of the map 
$G\to [x,G]:=\{[x,y]:y\in G\}$ given by  $y\mapsto [x,y]$
are cosets of $C_G(x)$ and hence they have  dimension  $\dim(G_G(x))\sim \dim(G)$. So, by Corollary \ref{big_col}(4) we get that $\dim([x,G])\in\omega$. Thus, for $x_1,x_2\in N$  the commutator $[x_1,x_2]$ has finite dimension over $a,x_1$ and over $a,x_2$, so  by Proposition \ref{lin_dim}
$[x_1,x_2]\in (\Lin_K(V(a,x_1))\cap \Lin_K(V(a,x_2)))^{k}$. If additionally $\dim(x_1/x_2,a)\sim\dim(x_1/a)$ then, by Corollary \ref{cor_int} we get that $$\Lin_K(V(a,x_1))\cap \Lin_K(V(a,x_2))=\Lin_K(V(a))=:A,$$ so $[x_1,x_2]\in A$.

Now, for arbitrary $y_1,y_2\in N$, as $\dim(y_2\cdot C_G(y_1))=\dim(C_G(y_1))\sim \dim(G)$, we can find $y_2'\in y_2\cdot C_G(y_1)$ with $\dim(y_2'/y_1,a)\sim \dim(G)$, so $[y_1,y_2]=[y_1,y_2']\in A$ by the above paragraph. This shows that $N_0:=\{[y_1,y_2]:y_1,y_2\in N\}\subseteq A$.
Put $N_1:=N\cap A$. As $A$ is definably closed by Fact \ref{dcl}(2), we get that $N_1$ is a definable subgroup of $N$. So, as $N_1\supseteq N_0$, we conclude that $N_1\supseteq N'$.  
Finally, put $$N_2:=\bigcap _{g\in N}(N_1)^g$$
and note that $N_2$ is a definable normal subgroup of $N$ and $N'\leq N_2$, so $N/N_2$ is abelian. Also, $N_2\leq N_1\subseteq A$. But, as $a$ is finite, we have that $A$ is finite-dimensional, so $N_2$ is   algebraic [semialgebraic] over $K(\C)$ by Corollary \ref{tfae}. 
Hence $N$ is [semi]algebraic-by-abelian, and $G$ is ([semi]algebraic-by-abelian)-by-[semi]algebraic.
\end{proof}

\begin{remark}\label{optimal}
 Examples \ref{examp}(1) and (2) show that the conclusion in Theorem \ref{thm_gps} cannot be strengthened  to `$G$ is [semi] algebraic-by-abelian', nor to `$G$ is abelian-by-[semi] algebraic'. Indeed, in Example \ref{examp}(1), taking  $H:=K^*$ acting naturally on $(V,+)$, we get that the commutator group $[G,G]=V\times \{1\}$ is infinite-dimensional, so $G$ is not [semi] algebraic-by-abelian. 
 On the other hand, the Heisenberg group $(G,\cdot)=(V\times V\times K,\cdot)$ in  Example \ref{examp}(2) is not abelian-by-[semi] algebraic. Indeed, working for example in $_ST^{ACF_p}_\infty$, if $N\triangleleft G$ is a normal definable subgroup such that $G/N$ is [semi] algebraic, then, by Corollary \ref{tfae} we get that $\dim(N)\sim \dim(G)=[2n+1]$. Hence, if $M=\bigcup^a_r N_r$ and $R\in \omega$ are such that $G$ and $H$ are definable over $N_R$, then  we can find $(v,w,a)\in H(M)$ with $v\notin N_R$ and $w\notin \Lin_{K(M)}(N_R,w)$. Let $v_0,v_1,w_0,w_1\in V(M)$ be such that $v=v_0+v_1$, $w=w_0+w_1$, $v_0,w_0\in V(N_R)$ and $v_1,w_1\perp V(N_R)$. As $[,]$ is nondegenerate and $v_1\notin V(N_R)$, there is $z\in V(M)$ with $[v_1,z]\neq 0$ and $z \perp V(N_R)$. Now we can choose $v',w''\in V(M)$
with $v',w''\perp v_1,w_1,z,V(N_R)$, $[v',v']=[v_1,v_1]$, and $[v',w'']=[v_1,w_1]$. Take any $e_0 \in V(M)$ with $e_0\perp v_1,w_1,z,v',w'',V(N_R)$ and $[e_0,e_0]=1$. Let $\alpha\in K(M)$ be such that $[\alpha e_0+w''+z,\alpha e_0+w''+z]=[w_1,w_1]$. Then putting $w':=\alpha e_0+w''+z$ we have  $[w',w']=[w_1,w_1]$, $[v',v']=[v_1,v_1]$, $[v',w']=\alpha [v',e_0]+[v',w'']+[v',z]=[v',w'']=[v_1,w_1]$, and $v_1,w_1,v',w'\perp V(N_R)$. Thus $\tp(v',w'/N_R)=\tp(v_1,w_1/N_R)$ and hence $\tp(v_0+v',w_1+w'/N_R)=\tp(v,w/N_R)$ so we can choose $b\in K$ with $\tp(v_0+v',w_1+w',b/N_R)=\tp(v,w,a/N_R)$. As $H$ is definable over $N_R$ and $(v,w,a)\in H$, we must have $(v_0+v',w_0+w',b)\in H$ as well. Now, the commutator $[(v,w,a),(v_0+v',w_0+w',b)]$ equals $(0,0,[v,w_0+w']-[v_0+v',w])=(0,0,[v_0,w_0]+[v_1,w']-([v_0,w_0]+[v',w_1]))=(0,0,[v_0,w_0]+[v_1,\alpha e_0]+[v_1,w'']+[v_1,z]-[v_0,w_0])=(0,0,[v_1,z])\neq (0,0,0)$. Hence $N$ is not abelian, and $G$ is not abelian-by-[semi] algebraic.
\end{remark}

Now we conclude from (the proof of) Theorem \ref{thm_gps} that all fields definable in $T^*_\infty$ have finite dimension.

\begin{theorem}
Every field definable in $T^*_\infty$ is finite-dimensional, and hence definable in the field of scalars $K$. In particular, there is no definable field structure on $V^k$ for any $k<\omega$.
\end{theorem}
\begin{proof}
Suppose $F$ is an infinite-dimensional field definable in $T^*_\infty$. Put $s:=\dim(F)$. 
 
 Let $G= (F^*,\cdot)\ltimes (F,+)$ be the affine group of $F$, that is, $G$ consists of pairs $(a,b)$ where $a\in F^*$ and $b\in F$ with multiplication given by:  $$(a,b)(c,d)=(ac,b+ad).$$ 
  Notice that for any $(a,b)(c,d)\in G$  the commutator $[(a,b),(c,d)]=(a,b)(c,d)(a,b)^{-1}(c,d)^{-1}$ is equal to $(1,(a-1)d+(1-c)b)$. Hence, if $a\neq 1$ and  $(c,d)\in C_G((a,b))$ then $d=\frac{c-1}{a-1}b$, so $\dim(C_G((a,b)))\leq s$, whereas $\dim(G)=2s$ by Corollary \ref{big_col}(4) applied to projection on either of the coordinates of the Cartesian product $F^*\times F$. Hence, if we put $N:=\{g\in G:\dim(G)-\dim(C_G(g))\in \omega\}$, we get that $N\subseteq \{1\}\times F$. This implies that the set $\{(a,0):a\in F^*\}$ embeds definably in $G/N$, so $\dim(N)\ll\dim(G)$ which contradicts the proof of Theorem \ref{thm_gps}.
\end{proof}
By Fact \ref{fact_acf}(2) we conclude:
\begin{corollary}
 Every infinite field definable in $T_\infty$ is definably isomorphic to the field of scalars $K$.
\end{corollary}
By Fact \ref{fact_omin_fields} we also get:
\begin{corollary}
 Every infinite field definable in $T^{RCF}_\infty$  is either algebraically closed or real closed.
\end{corollary}

\section{Independence relations and generics}\label{sect_generics}
In this section we relate our notion of dimension in $T_\infty$ to two independence relations, $\ind^\Gamma$ introduced in \cite[Definition 12.2.1]{Gr} for $T_\infty$, and Kim-independence (denoted $\ind^K$) defined for any theory in \cite{KR}, and having good properties over models in NSOP$_1$ theories, and over arbitrary sets in NSOP$_1$ theories satisfying existence.

We will work in a monster model $\C\models T_\infty$, that is, a $\bar{\kappa}$-saturated, $\bar{\kappa}$-strongly homogeneous model of $T_\infty$ for some sufficiently large $\bar{\kappa}$. All parameter sets considered will be \emph{small}, that is, of size less than $\bar{\kappa}$.

We say that a set $A$ is an \emph{extension base} if no formula (or equivalently, type) over $A$ forks over $A$. We say that a theory $T$ satisfies the \emph{existence axiom} (or simply \emph{existence}) if every set of parameters is an extension base.
It was asked in \cite[Question 6.6]{DKR} whether any NSOP$_1$ theory satisfies existence, and a list of positive examples was given in \cite[Fact 2.14]{DKR}. Here we show that $T_\infty$ also satisfies it:
\begin{proposition}\label{prop_existence}
 $T_\infty$ satisfies existence. 
\end{proposition}
\begin{proof}
Let $\phi(x,a)$ be a formula over $A$. Let $p(x)$ be a global generic type in $\phi(x,a)$. As any conjugate of $p(x)$ over $A$ is also a generic type in $\phi(x,a)$, we get by Corollary \ref{cor_mlt}(1) that there are only finitely many conjugates of $p(x)$ over $A$. As $p(x)$ is definable by Corollary \ref{cor_mlt}(2), this implies that it is definable over $\acl^{eq}(A)$; in particular, $p(x)$ does not fork over $\acl^{eq}(A)$, so it does not fork over $A$, so $\phi(x,a)$ does not fork over $A$.
\end{proof}

\begin{fact}\cite[Theorem 12.2.2]{Gr}\label{fact_gam}
Let $M\models T_\infty$. Then the relation $\ind^{\Gamma}$ on subsets of $M$ given by $\Gamma$-forking is automorphism invariant, symmetric, transitive, satisfies the finite character and extension axioms, and types over models are stationary. 
\end{fact}

Below, if $p(x)\in S(B)$ is a complete type in $T_\infty$ and $B\subseteq N_R\models T_R$, then we say that $p(x)$ forks in $N_R$ over some $A\subseteq B$ if its quantifier-free part in the language $L^F_\theta$ (which is equivalent to $p(x)$ in $T_\infty$) forks in $N_R$ over $A$.  
Likewise, $\RM_{N_R}(p(x))$ means Morley rank of the quantifier-free part of $p(x)$ in the sense of $N_R$.
\begin{fact}\cite[Proposition 12.2.3]{Gr}\label{fact_gamma}
 Let $M\models T_\infty$, let $A\subseteq B\subseteq M$ and let $p(x)\in S(B)$. Let $(N_r:r\in\omega)$ be some approximating sequence for $M$. Then the following are equivalent:\\
 (1) $p(x)$ does not $\Gamma$-fork over $A$;\\
 (2) Given any formula $\phi=\phi(x,b)\in p(x)$ there is  $R_\phi\in \omega$ such that $\phi(x,b)$ does not fork over $A\cap N_r$ in the structure $N_r$ for all $r\geq R_\phi$;\\
 (3) For each finite $b\subseteq B$ there is $R_b\in \omega$ such that $p(x)|_{N_r\cap Ab}$ does not fork over $A\cap N_r$ in $N_r$ for all $r\geq R_b$. 
\end{fact}

\begin{corollary}\label{cor_fork}
Let $M\models T_\infty$, let $A\subseteq B\subseteq M$ with $A$ finite, and let $p(x)\in S_{<\omega}(B)$.
Then $p(x)$ does not $\Gamma$-fork over $A$ if and only if $\dim(p(x))=\dim(p|_A(x))$.
\end{corollary}
\begin{proof}
Assume $\dim(p(x))=\dim(p|_A(x))$. We will verify that the condition (3) in Fact \ref{fact_gamma} holds for $p(x)$ and $A$.  Consider any finite $b\subseteq B$ and  $M=\bigcup^a_rN_r$ containing $Ab$. Let $R\geq \alpha(l(x),l(Ab))$ be such that $Ab\subseteq N_{R}$. Consider any $r>R$. Note that $\dim(\phi(x))\in D_{l(x),l(Ab)}$ for any formula $\phi(x)$ with parameters in $Ab$. Hence, as $\dim(p(x))=\dim(p|_A(x))$, we have by Lemma \ref{lemma_geq} that 
 $$\RM_{N_r}(p|_{Ab}(x))=\RM_{N_r}(p|_A(x)),$$ so $p|_{Ab}(x)$ does not fork over $A$ in $N_r$, and  by Fact \ref{fact_gamma} we get that $p|_{Ab}(x)$ does not $\Gamma$-fork over $A$, as required.
 
 Similarly, if $\dim(p(x))<\dim(p|_A(x))$ then there is a finite $b$ such that 
 $\dim(p|_{Ab}(x))<\dim(p|_A(x))$, and if $M=\bigcup^a_r N_r$ and $R\geq \alpha(l(x),l(Ab))$ are such that $Ab\subseteq N_R$, then we have by Lemma \ref{lemma_geq} that $\RM_{N_r}(p|_{Ab}(x))<\RM_{N_r}(p|_A(x))$, so,
 by Fact \ref{fact_gamma}, $p|_{Ab}(x)$ $\Gamma$-forks over $A$.
\end{proof}

\begin{definition}\cite[Definition 1.11]{EKP}\label{def_gen} Let
$A\subseteq B\subseteq \C\models T$ for some theory $T$, and let $G$ be a group definable in $\C$ over parameters $A$. Suppose $\ind^*$ is an invariant ternary relation between small subsets of $\C$. 
We call an element $g\in G$ a (left) \emph{generic} over $B$, if 
$$\mbox{for every }h\in G\mbox{ such that }g\ind^*_B h \mbox{ we have }h\cdot g\ind_A^* B,h.$$ We call a type $p(x)\in S_G(B)$ (left) generic in $G$ if every (equivalently, some) its realisation is a generic in $G$ over $B$. 
\end{definition}
This notion of a generic was first studied in groups definable in stable (e.g. \cite{Poi}), and more generally, simple theories (e.g. \cite{Pil}), with $\ind^*$ being the forking independence. In \cite{EKP} it was studied in rosy theories mainly with $\ind^*$ being thorn-independence. In the  (non-first order) setting of Polish group structures a useful notion of a generic is obtained by taking $\ind^*$ to be the non-meagre independence (\cite{Kr}).
Below we examine this notion in $T_\infty$  for $\ind^*=\ind^\Gamma$ and for $\ind^*=\ind^K$.

To avoid confusion,  we will say `$\dim$-generic' to mean  generic in the sense
of Definition \ref{def_dimgen}.

\begin{proposition}
Suppose $A\subseteq B\subseteq \C\models T_\infty$, where $A$ is finite, and $G$ is a group definable  over $A$.  Then for any $p(x)\in S(B)$ we have that $p(x)$ is $\ind^\Gamma$-generic in $G$ if and only if $p(x)$ is $\dim$-generic in $G$. 
\end{proposition}
\begin{proof}
 Suppose first $p(x)$ is $\dim$-generic in $G$ (i.e. $\dim(p(x))=\dim(G)$) and fix any $g\models p$ and $h\in G$ such that $g\ind^{\Gamma}_B h$. Then by Corollary \ref{cor_fork} $\dim(\tp(g/B,h))=\dim(\tp(g/B))=\dim(p(x))=\dim(G)$. As $\dim$ is preserved by definable bijections and every formula in $q:=\tp(h\cdot g/B,h)$ is a translate of a formula in $\tp(g/B,h)$, we conclude that $\dim(q(x))=\dim(G)$. On the other hand, $q|_A\vdash G(x)$, so $\dim(q|_A(x))\leq \dim(G)$, so we must have $\dim(q(x))=\dim(q|_A(x))$. By Corollary \ref{cor_fork} again, this gives that $q(x)$ does not $\Gamma$-fork over $A$, i.e. $h\cdot g\ind_A^\Gamma B,h$.
 
 Now suppose $p=\tp(g/B)$ is a $\ind^\Gamma$-generic in $G$. By Proposition \ref{rk_extension}  we can find $h\in G$ with $\dim(h/B,g)=\dim(G)$. In particular, $g\ind^{\Gamma}_B h$. 
 As $g$ is generic in $G$ over $B$, we have $h\cdot g\ind_A^\Gamma B,h$ so $h\cdot g\ind_B^\Gamma h$. Using this together with Corollary \ref{cor_fork} in the second equality below, we get:
  $$\dim(g/B)\geq \dim(g/B,h)=\dim(h\cdot g/B,h)=$$ $$=\dim(h\cdot g/B)\geq \dim(h\cdot g/B,g)=\dim(h/B,g)=\dim(h/B)=\dim(G).$$
 Clearly $\dim(g/B)\geq \dim(G)$ implies that $\dim(g/B)= \dim(G)$, as $\tp(g/B)\vdash G(x)$.
 \end{proof}

\begin{corollary}
For any group $G$ definable in $T_\infty$ over a finite set $A$ and any $B\supseteq A$ there exists a $\ind^\Gamma$-generic over $B$ element in $G$, and being $\ind^\Gamma$-generic does not depend on the choice of the finite set $A$ over which $G$ is definable.
\end{corollary}

Kim-independence was introduced and studied extensively in \cite{KR} over models in NSOP$_1$ theories. It was proved there, among other results, that $\ind^K$ is symmetric and satisfies the independence theorem over models, which was later extended in \cite{DKR} to arbitrary sets in NSOP$_1$ theories satisfying existence.
\begin{definition}\cite[Definition 2.10]{DKR}
\begin{enumerate}
 \item We say a formula
$\varphi(x,a_0)$ {\em Kim-divides} over $A$ if for some Morley sequence $\langle a_i : i<\omega\rangle$ in $\tp(a_0/A)$,
$\{ \varphi(x, a_i)|\ i<\omega\}$ is inconsistent. 
\item A formula $\varphi(x;a)$ {\em Kim-forks} over $A$ if $\varphi(x;a) \vdash \bigvee_{i < k} \psi_{i}(x;b_{i})$ where $\psi_{i}(x;b_{i})$ Kim-divides over $A$ for all $i < k$.  
\item Likewise we say a type $p(x)$ Kim-forks or Kim-divides over $A$ if it implies a formula that Kim-forks or Kim-divides over $A$, respectively.  
\item We write $a \ind^{K}_{A} b$ to denote the assertion that $\tp(a/Ab)$ does not Kim-fork over $A$.  
%
%
\end{enumerate}
\end{definition}

\begin{fact}\cite{DKR}
Suppose $T$ is NSOP$_1$ and satisfies existence. Then:\\
(1) Kim's Lemma holds for $\ind^K$, that is, if a formula $\phi(x,a_0)$ Kim-divides over $A$ then for every Morley sequence $\langle a_i : i<\omega\rangle$ in $\tp(a_0/A)$,
$\{ \varphi(x, a_i)|\ i<\omega\}$ is inconsistent\\
(2) A formula Kim-forks over $A$ if and only  if it Kim-divides over $A$\\
(3) $\ind^K$ is symmetric\\
(4) The independence theorem for Lascar types for $\ind^{K}$ holds over any set.
\end{fact}

The following folklore fact follows as in \cite[Corollary 5.17]{KR}, using the fact that, for any set $C$, a sequence is Morley over $C$ iff it is a Morley sequence over $\acl(C)$.
\begin{fact}\label{kimacl}
 Suppose $T$ is an NSOP$_1$ theory satisfying existence, and let $A$, $B$, and $C$ be any sets. Then $A\ind^K_C B\iff \acl(A)\ind^K_{\acl(C)}\acl(B)$. Also, it follows from the definition of $\ind^K$ that $A\ind^K_C B$ implies $A\ind^K_C BC$.
\end{fact}

We will now give a description of Kim-independence in $T_\infty$ over arbitrary sets. The proof of it will be essentially the same as in the description of Kim-independence over models given in \cite[Proposition 9.37]{KR}, but the statement there requires two corrections (even when working over a model), which we now explain. If  $A\subseteq \C\models T_\infty$, put
$\langle A \rangle:=\Lin_{K(\C)}(V(A))$ and  let $A_K:=A\cap K(\C)$. By $\acl(A)_K$ we mean $(\acl(A))_K$, where $\acl$ is the model-theoretic algebraic closure.

By quantifier elimination the structure on the sort $K$ induced from $T_\infty$ is just the pure field structure, so the relation $\ind^K$ restricted to the sort $K$ coincides with forking independence $\ind^{ACF}$ in the algebraically closed field $K$, that is, with algebraic independence in the sense of field theory. As for algebraically closed $ A,B\supseteq M\models T_\infty$  the condition $A\cap B=M$ does not imply that $K(A)\ind^{ACF}_{K(M)}K(B)$, the latter is a missing condition in  \cite[Proposition 9.37]{KR}. 

Also, the following example shows that the condition $A\cap B=M$ for algebraically closed sets $A, B$ and a model $M$, does not imply  that $\langle A\rangle\cap \langle B\rangle=\langle M\rangle$ (even if $K_A=K_B=K_M$), which is also implicitly used in the proof of \cite[Proposition 9.37]{KR}, and which clearly follows from $A\ind^K _M B$ (see the proof of Proposition \ref{kim_description} below).

\begin{example}\label{ex_lin}
 Let $M=(V_0,K_0)\models {_ST_\infty}$ and choose an orthonormal pair $(e_0,e_1)$ with   $e_0,e_1\in M^\perp$. Put $A:=(\Lin_{K_0}(M,e_0,e_1),K_0)$. Clearly $A=\acl(A)$.  
 Let $t\in K(\C)\backslash K_0$ and let $t'$ be such that
 $t^2+t'^2=1$. Put $f:=te_0+t'e_1$ and  
 $B:=(\Lin_{K_0}(M,f),K_0)$. As $[f,f]=1$ and $f\in M^\perp$, we have $B=\acl(B)$. 
 Clearly $\langle A \rangle \cap \langle B\rangle=\langle B\rangle\neq M$, but
 $A\cap B=M$: any element of $A\cap B$ is of the form $ae_0+be_1+m_0=cf+m_1$ for some $a,b,c\in K_0$ and $m_0,m_1\in V(M)$. Then, as $\langle e_0,e_1,f \rangle\subseteq M^\perp$, we have 
  $$ae_0+be_1=cf=cte_0+ct'e_1$$
 so $a=ct$ and $b=ct'$. As $t,t'\notin K$, this implies that $a=b=0$, so $ae_0+be_1+m_0=m_0\in M$. Hence $A\cap B=M$.
 \end{example}


\begin{proposition}\label{kim_description}
Let
 $A,B,E\subseteq \C\models T_\infty$ be small algebraically closed sets with $E\subseteq A,B$. Then $A\ind^K_E B$ if and only if $\langle A\rangle \cap \langle B\rangle=\langle E\rangle $ and $K(A)\ind ^{ACF}_{K(E)}K(B)$.
\end{proposition}
\begin{proof}
Suppose first that $A\ind^K_E B$. As already pointed out above, this implies that $K(A)\ind ^{ACF}_{K(E)}K(B)$ and it is left to show that $\langle A\rangle \cap \langle B\rangle=\langle E\rangle $. Suppose this is not the case, so there are vectors $a_1,\dots,a_m\in A$ and $b_1,\dots, b_k\in B$ such that $\langle a_1,\dots,a_m\rangle\cap \langle b_1,\dots,b_k\rangle$ is not contained in $\langle E \rangle$. Put $E_0:=E\cap \langle b_1,\dots,b_k\rangle$; then $E_0$ is a finite-dimensional vector subspace of $\langle E\rangle$. Hence there is a formula $\phi(x_1,\dots,x_m,b_1,\dots,b_k)$ over $E_0b_1\dots b_k$ expressing that $\langle x_1,\dots,x_m \rangle\cap \langle b_1,\dots, b_k \rangle$ is not contained in $E_0$. Then $\models \phi(a_1,\dots,a_m,b_1,\dots,b_k)$ and $\phi(x_1,\dots,x_m,b_1,\dots,b_k)$ Kim-divides over $E$, as for any Morley sequence $(\bar{d_i})_{i\in \omega}$ in $\tp(b_1,\dots ,b_k/E)$ and $i\neq j$ we have  $\langle \bar{d_i}\rangle\cap \langle \bar{d_j}\rangle= E_0$, so the set of formulas $\{\phi(x_1,\dots,x_m,\bar{d_i}):i\in \omega\}$ is 2-inconsistent.
This is a contradiction, hence the implication from left to right is proved.

Let us now assume that $\langle A\rangle \cap \langle B\rangle=\langle E\rangle $ and $K(A)\ind ^{ACF}_{K(E)}K(B)$.
There are only two problems with the proof of Proposition 9.37 in \cite{KR} (with $E=M$ a model).
First, as shown by Example \ref{ex_lin}, the assumptions there do not imply that $\langle A\rangle \cap \langle B \rangle=\langle M\rangle$, which is used in the construction of the structure $N$ in that proof. This is, however, assumed here. Secondly,  in the last paragraph of the proof in \cite{KR}, the map $\sigma^i:B_0\to B_i$ 
need not be elementary over $K(A')=K(A)$ on the sort $K$. However,
assuming that $K(A)\ind^{ACF}_{K(M)} K(B)$ we have that $K(B_i)\ind^{ACF}_{K(M)}K(A')$ and $K(B_0)\ind^{ACF}_{K(M)}K(A')$, so  the map $\id_{K(A')}\cup \sigma^i|_{K(B_0)}:K(A')\cup K(B_0)\to K(A')\cup K(B_i)$ is elementary by stationarity of $\tp(K(B_0)/K(M))$. Thus $\id_{K(A')}\cup \sigma^i|_{K(B_0)}$ extends to an isomorphism $\rho:\tilde{K}\to L$ onto some algebraically closed field $L$, hence, by the construction of $A'$, the map $\id_{A'}\cup \sigma^i$ extends to an isomorphism between $\Lin_{\tilde{K}}(A'\cup B_0)$ and $\Lin_{L}(A'\cup B_i)$. By quantifier elimination this isomorphism is an elementary map, so in particular $A'B_0\equiv A'B_i$, and hence $\tp(A/B)$ does not Kim-divide over $M$. 

When $E$ is an arbitrary algebraically closed set (not necessarily a model), the only difference is that the vector spaces we obtain may be finite-dimensional, which does not cause any problems, as an isomorphism of finite-dimensional vector subspaces of $\C$ is still elementary in $T_\infty$ by quantifier elimination. Hence the implication from right to left holds.
\end{proof}
By \cite[Proposition 9.5.1]{Gr}
$\acl(AC)_K$ is the field-theoretic algebraic closure of $\dcl(AC)_K$ and $\langle\acl(AC)\rangle=\langle AC\rangle$ for any sets $A$ and $C$, so by Fact \ref{kimacl} we conclude:
\begin{corollary}\label{kim_description2}
Let $A,B,C\subseteq \C\models T_\infty$ be any small sets. Then $A\ind^K_C B$ if and only if $\langle AC \rangle\cap \langle BC \rangle =\langle C \rangle$ and $\dcl(AC)_K\ind ^{ACF}_{\dcl(C)_K}\dcl(BC)_K$. 
\end{corollary}
 

As $\ind^K$ does not satisfy base monotonicity, it is not obvious whether in the definition of $\ind^K$-genericity over $B$ of an element $g\in G$ with $G$ definable over $A$ it is more reasonable to require that $h\cdot g\ind^K_A B,h$ (as is done in Definition \ref{def_gen})  or that $h\cdot g\ind^K_B h$, provided that $ g\ind^K_B h$. In either case, it turns out that $(V,+)$ does not have any $\ind^K$-generics over any set of parameters. Below we prove it for $\ind^K$-genericity in the sense of Definition \ref{def_gen}, and exactly the same argument works for the other sense.
\begin{proposition}
 The  $\emptyset$-definable in $T_\infty$ group $(V,+)$ does not have any $\ind^K$-generic type over any set $B$.
\end{proposition}
\begin{proof}
As usually we consider the symmetric case, the alternating case being very similar.
 By Fact \ref{kimacl} being $\ind^K$-generic over $B$ is the same as being $\ind^K$-generic over $\acl(B)$, so we may assume that $B=\acl(B)=(V_0,K_0)$; in particular, $K_0$ is an algebraically closed field. Consider any $v\in V$ and put $(V_1,K_1)=\acl(B,v)$ and $a=[v,v]$. We will show that $v$ is not a $\ind^K$-generic in $(V,+)$ over $B$. 
  If $v\in \langle V_0\rangle$ then for any $w\neq v$ with $v\ind^K_B w$ we have that $0\neq w+v\in \langle  w+v\rangle \cap \langle V_0,w\rangle$, so $w+v\nind^K B, w$ hence $v$ is not a $\ind^K$-generic in $(V,+)$ over $B$.
  So let us assume that  $v\notin \langle V_0\rangle$. Let $t\in K(\C)\backslash K_0$ be such that 
 $K_1\ind^{ACF}_{K_0} t$.
 \begin{claim}
  We may assume there exists $w\in V$ such that $w\perp V_0$, $[w,w]=t$, $[w,v]=-\frac 1 2 a$ and $\langle V_0,v \rangle\cap \langle V_0,w \rangle =\langle V_0\rangle$. 
 \end{claim}
\bpfc
 As $v\notin \langle V_0\rangle$, by compactness and the Gram-Schmidt process we can easily find some $f\in V$ with $f\perp V_0$ and $[f,v]=-\frac 1 2 a$. Let $e_1\in V$ be orthogonal to $\langle V_1,f \rangle$ with $[e_1,e_1]=1$. Now we can find $\beta \in K(\C)$ such that $[f,f] +\beta^2=t$. Then putting $w:=f +\beta e_1$ we get $[w,w]=[f,f]+\beta^2=t$ and $[w,v]=[f,v]=-\frac 12a$. By possibly modifying $t$ we may assume that $\beta \neq 0$, so $\langle V_0,v \rangle\cap \langle V_0,w \rangle =\langle V_0\rangle$.
\epf
Let $w$ be as in the claim. Then $[w+v,w+v]=[w,w]+[v,v]+2[w,v]=t+a-a=t=[w,w]$, so  $ w+v\nind^K B,w$. On the other hand, as $w\perp V_0$, we have by \cite[Proposition 9.5.1]{Gr} that $\dcl(B,w)_K=\dcl_{ACF}(K_0,t)$. As $K_1=\acl(B,v)_K$, this gives us that $\dcl(B,v)_K\ind^{ACF}_{K_0}\dcl(B,w)_K$ by the choice of $t$. 
As we also know by the choice of $w$ that
$\langle V_0,v \rangle\cap \langle V_0,w \rangle =\langle V_0\rangle$, we conclude that $v\ind^K_B w$. Hence $v$ is not a $\ind^K$-generic in $(V,+)$ over $B$.
\end{proof}

\begin{question}
 Is there a useful notion of a generic element in a group definable in an NSOP$_1$ theory with existence?
\end{question}

\end{document}